\documentclass[a4paper,10pt]{article}
\usepackage{geometry}
\usepackage{epsfig}
\usepackage{amsmath}
\usepackage{amssymb}
\usepackage{amsthm}
\usepackage{mathrsfs}
\usepackage{color}
\usepackage{amscd}
\usepackage{euscript}
\usepackage{ stmaryrd }
\usepackage{afterpage}
\usepackage{authblk}
\usepackage{enumitem}

\usepackage[backend=bibtex,style=numeric]{biblatex}
\usepackage{ stmaryrd }

\theoremstyle{plain} 
\newtheorem{thm}{Theorem}[section] 
\newtheorem{cor}[thm]{Corollary} 
\newtheorem{lem}[thm]{Lemma} 
\newtheorem{prop}[thm]{Proposition} 
\theoremstyle{definition} 
\newtheorem{defn}{Definition}[section] 
\newtheorem{oss}{Remark}

\newtheorem{ex}{Example}

\DeclareFontEncoding{LS1}{}{}
\DeclareFontSubstitution{LS1}{stix}{m}{n}
\DeclareSymbolFont{symbols2}{LS1}{stixfrak} {m} {n}
\DeclareMathSymbol{\operp}{\mathbin}{symbols2}{"A8}

\title{{\bf  Almansi-type decomposition for slice regular functions of several quaternionic variables }}

\author{ Giulio Binosi\footnote{ORCID: \texttt{0000-0002-4733-6180}} \\
\small Dipartimento di Matematica, Universit\`a di Trento\\ 
\small Via Sommarive 14, I-38123 Povo Trento, Italy\\
\small giulio.binosi@unitn.it
}

\date{ {\small
\textit{
}} }

\begin{document}

\maketitle

\begin{abstract}
In this paper we propose an Almansi-type decomposition for slice regular functions of several quaternionic variables. Our method yields $2^n$ distinct and unique decompositions for any slice function with domain in $\mathbb{H}^n$. 
Depending on the choice of the decomposition, every component is given explicitly, uniquely determined and exhibits desirable properties, such as harmonicity and circularity in the selected variables. As consequences of these decompositions, we give another proof of Fueter's Theorem in $\mathbb{H}^n$, establish the biharmonicity of slice regular functions in every variable and derive mean value and Poisson formulas for them.
\end{abstract}

\section{Introduction}
In 1899, Emilio Almansi \cite{AlmansiClassico} proved that any polyharmonic function $f$ of degree $m$, defined on a star-like domain centered at the origin of $\mathbb{R}^n$ could be written as a combination of $m$ harmonic functions $\{\mathcal{S}_i(f)\}_{i=0}^{m-1}$ as
\begin{equation*}
    f(x)=\mathcal{S}_0(f)(x)+|x|^2\mathcal{S}_1(f)(x)+\dots+|x|^{2(m-1)}\mathcal{S}_{m-1}(f)(x).
\end{equation*}
This theorem plays a central role in the theory of polyharmonic functions, establishing a bridge between the theories of harmonic and polyharmonic functions. We refer to the monograph \cite{PolyharmonicFunctions} and the references therein for applications of the theorem in the classical case of several real or complex variables.

Generalizations of Almansi decomposition have been studied both concerning other type of iterated differential operators (e.g. in \cite{AlmansiDunkl} for the Dunkl Laplacian), both for more general classes of functions, especially in the hypercomplex settings of slice regular \cite{Almansi,AlmansiPerottiClifford} and monogenic functions \cite{AlmansiPolymonogenic}. But, as far as we know, no Almansi decomposition has been provided for functions of several hypercomplex variables.

Our starting point is the work of A. Perotti \cite{Almansi}, in which an Almansi-type decomposition holds for slice regular functions of one quaternionic variable:
\begin{thm}\cite[Theorem 4]{Almansi}
\label{teorema intro perotti}
    Let $f$ be a slice regular function defined on a circular set $\Omega\subset\mathbb{H}$, then there exist two unique, circular and harmonic functions $h_1,h_2$ such that
    \begin{equation*}
        f(x)=h_1(x)-\overline{x}h_2(x),\qquad\forall x\in\Omega.
    \end{equation*}
    The functions are given by $h_1=(xf)'_s=-\overline{\partial}_{CRF}(xf)$ and $h_2=f'_s=-\overline{\partial}_{CRF}(f)$,
    where $g'_s(x):=[2\operatorname{Im}(x)]^{-1}(g(x)-g(\overline{x}))$ is the spherical derivative of a slice function $g$, which, up to a factor, agrees on slice regular functions with the Cauchy-Riemann-Fueter operator $\overline{\partial}_{CRF}=\frac{1}{2}(\partial_{\alpha}+i\partial_\beta+j\partial_\gamma+k\partial_\delta)$.
\end{thm}
The aim of the present paper is to extend the previous Almansi-type decomposition of slice regular functions to the context of several variables.
The extension to higher dimensions poses some new challenges, one of which is the exponential growth of all possible decompositions. Indeed, for slice functions of $n$ variables, we obtain $2^n$ decompositions (Theorem \ref{Teorema principale}), as the cardinality of all possible choices of variables between $x_1,...,x_n$.
Every component is radially symmetric in the imaginary part of the chosen variables that determine the decomposition; if, moreover, the decomposing function is slice regular they are harmonic in the same variables, too.
Every component of each decomposition is given explicitly and it is completely determined by the original function through its partial spherical derivatives, as in Theorem \ref{teorema intro perotti}.
We also prove the unique character of these decompositions (Proposition \ref{Proposizione unicità decomposizione}), namely the functions performing the decomposition are unique, if specific symmetry properties are required.

Among these, we point out the class of ordered decompositions, corresponding to integers intervals of the form $\{1,2,...,m\}$ (Corollary \ref{Corollario Almansi ordinato}). 
Such special components are sufficient for characterizing the slice regularity of the slice function they decompose (Proposition \ref{Proposizione caratterizzazione slice regolarita con componenti ordinate}), in analogy with the one variable interpretation of slice regularity proposed in \cite[\S 3.4]{Several}. In this case they can also be given applying iteratively the Cauchy-Riemann-Fueter operator $\overline{\partial}_{x_h}$, instead of the partial spherical derivatives. The components of the ordered decompositions have already been exploited to define strongly slice regular functions of several variables, which are a generalization of slice regular functions defined on a not necessarily axially symmetric domain \cite{PerottiWirtinger}.

We describe the structure of the paper. In Section 2 we recall from \cite{Parteteorica} and \cite{Several} the main features of slice regular functions of several quaternionic variables and some properties of partial spherical values and derivatives. Section 3 is divided into three parts: in the first one we state the main results described above, whose proofs are postponed to \S 3.3, while in the second one we show some preliminary results.

Last section is devoted to applications of Almansi-type decomposition. First, we apply Theorem \ref{Teorema principale} to slice regular polynomials, where the components are given through zonal harmonics, namely, rotational invariant harmonic functions. Then, by the harmonicity of the components of Almansi-type decompositions, we give a different proof of the several variables version of Fueter's Theorem (\cite[Theorem 4.10]{Parteteorica}) and prove the biharmonicity of any slice regular function w.r.t. each variable (Corollary \ref{Corollario biarmonicita}), justifying a posteriori the existence of an Almansi-type decomposition. Moreover, we are able to give examples of monogenic functions of several variables, in the spirit of Fueter's Theorem (Proposition \ref{proposizione componenti ordinate monogeniche}) generated from the components of Almansi-type decomposition. Finally, we derive various formulas (\S \ref{Sezione integrali}), such as mean value and Poisson formulas for slice regular functions.

\section{Preliminaries}
The study of quaternionic analysis has led to the development of various classes of functions that generalize complex analysis to higher dimensions. Among them, the notion of slice regularity has received particular attention in the last years. 
Slice analysis was firstly introduced by Gentili and Struppa in \cite{GentiliStruppa}, where they defined a new class of regular quaternionic functions as real differentiable functions which are holomorphic in every complex slice of $\mathbb{H}$. They initially referred to that class as C-regular functions, in honor of Cullen \cite{Cullen}, who previously conceived that definition. We refer the reader to \cite{LibroCaterina},\cite{NoncommutativeFunctionalCalculus}, \cite{entiresliceregularfunctions} and \cite{SlicehyperholomorphicSchuranalysis} for a comprehensive treatment of the theory of slice regular functions of one quaternionic variable.

Various attempts have been made to generalize this new hypercomplex theory and more general algebras (see e.g. \cite{Octonionsetting} and \cite{Cliffordsetting}) has been considered. Then, a new approach based on the concept of stem functions was introduced in \cite{SRFonAA} by Ghiloni and Perotti, using an idea that goes back to Fueter \cite{Fueter}. This gave slice analysis a crucial development extending the theory to any real alternative $^\ast$-algebra with unity, embodying the aforementioned generalizations, as well as the quaternionic case. Moreover, from this formulation, it was possible to give a definition of slice function without requiring any regularity assumption. 
The stem function approach paved the way to achieve an analogous theory in several variables \cite{Several}, which is now of great interest (see e.g. \cite{ghiloni2011slice},\cite{ghiloni2012slice},\cite{Parteteorica},\cite{colombo2012algebraic},\cite{gori2022zero}, \cite{DouSeveral}, and \cite{cervantes2023some}). Here we recall from \cite{Several} and \cite{Parteteorica}, the main features of the theory of slice regular functions of several variables, focusing on the quaternionic case.

\subsection{Slice regular functions of several quaternionic variables}
For any $n\in\mathbb{N}$, let $\mathcal{P}(n):=\mathcal{P}(\{1,...,n\})$, with $\mathcal{P}(U):=\{V\subset U\}$. Given $K=\left\{k_1,...,k_p\right\}\in\mathcal{P}(n)$, with $k_1<...<k_p$, and $(q_{k_1},...,q_{k_p})\subset\mathbb{H}^p$, denote the usual ordered $\mathbb{H}$-product of its elements by  $q_K:=q_{k_1}\cdot...\cdot q_{k_p}$ (if $K=\emptyset$, we set $q_\emptyset:=1$). 
Let $\mathbb{S}_\mathbb{H}$ be the set of square roots of $-1$ in $\mathbb{H}$, i.e.
$    \mathbb{S}_\mathbb{H}:=\{J\in\mathbb{H}\mid J^2=-1\}\subset\operatorname{Im}\left(\mathbb{H}\right)
$. Note that every element $q\in\mathbb{H}\setminus\mathbb{R}$ can be uniquely represented as $q=\alpha+J\beta$, with $\alpha\in\mathbb{R}$, $\beta\in\mathbb{R}^+$ and $J\in\mathbb{S}_\mathbb{H}$. For such $q=\alpha+J\beta$, define $\mathbb{S}_q=\mathbb{S}_{\alpha,\beta}:=\{\alpha+J\beta\mid J\in\mathbb{S}_\mathbb{H}\}\subset\mathbb{H}$. For any $h\in\{1,...,n\}$, let $\mathbb{R}_h:=\{(x_1,...,x_n)\in\mathbb{H}^n\mid x_h\in\mathbb{R}\}$ and if $H\in\mathcal{P}(n)$, let $\mathbb{R}_H:=\bigcap_{h\in H}\mathbb{R}_h$.

A continuous function $F=\sum_{K\in\mathcal{P}(n)}e_KF_K:D\subset\mathbb{C}^n\rightarrow\mathbb{H}\otimes\mathbb{R}^{2^n}$ is called stem function if the set $D$ is invariant under complex conjugation, (i.e. $z\in D$ if and only if $\overline{z}^h\in D$, where $\overline{z}^h:=(z_1,...,z_{h-1},\overline{z}_h,z_{h+1},...,z_n)$) and if $F_K$ satisfies
\begin{equation*}
    F_K(\overline{z}^h)=(-1)^{|K\cap\{h\}|}F_K(z),\qquad\forall K\in\mathcal{P}(n),h=1,...,n.
\end{equation*}
The circularization $\Omega_D\subset\mathbb{H}^n$ of an invariant set $D$ is defined as $\Omega_D:=\{(\alpha_1+J_1\beta_1,...,\alpha_n+J_n\beta_n)\mid (\alpha_1+i\beta_1,...,\alpha_n+i\beta_n)\in D, J_1,...,J_n\in\mathbb{S}_\mathbb{H}\}$. Any stem function $F$ uniquely induces a slice function $f:=\mathcal{I}(F):\Omega_D\rightarrow\mathbb{H}$, defined for every $x=(\alpha_1+J_1\beta_1,...,\alpha_n+J_n\beta_n)\in\Omega_D$ as
\begin{equation*}
f(x):=\displaystyle\sum_{K\in\mathcal{P}(n)}[J_K,F_K(z)],
\end{equation*}
where $z=(\alpha_1+i\beta_1,...,\alpha_n+i\beta_n)\in D$ and $[J_K,F_K(z)]:=J_KF_K(z)=J_{k_1}...J_{k_p}F_{\{k_1,...,k_p\}}(z)$ if $K=\{k_1,...,k_p\}$, with $k_1<...<k_p$. $Stem(D)$ will denote the set of stem functions $F:D\rightarrow\mathbb{H}\otimes\mathbb{R}^{2^n}$, $\mathcal{S}(\Omega_D)$ the set of slice functions $f:\Omega_D\rightarrow\mathbb{H}$ and $\mathcal{I}:Stem(D)\rightarrow\mathcal{S}(\Omega_D)$ the map sending a stem function to its induced slice function.
A slice function is called slice preserving whenever the components of its inducing stem function are real valued, we use the symbol $\mathcal{S}_\mathbb{R}(\Omega_D)$ to denote the set of slice preserving functions.

We can define a product between stem functions: given two stem functions $F$ and $G$, with $F=\sum_{K\in\mathcal{P}(n)}e_KF_K$, $G=\sum_{H\in\mathcal{P}(n)}e_HG_H$, define 
\begin{equation*}
    (F\otimes G)(z):=\displaystyle\sum_{H,K\in\mathcal{P}(n)}(-1)^{|H\cap K|}e_{K\Delta H}F_K(z)G_H(z),
\end{equation*}
with $K\Delta H:=(K\cup H)\setminus(K\cap H)$. The product of stem functions is again a stem function, \cite[Lemma 2.34]{Several} and this allows to define a product between slice functions: given $f=\mathcal{I}(F),g=\mathcal{I}(G)\in\mathcal{S}(\Omega_D)$, define $f\odot g:=\mathcal{I}(F\otimes G)$. Thus, $\mathcal{I}:(Stem(D),\otimes)\rightarrow(\mathcal{S}(\Omega_D),\odot)$ is an algebra isomorphism.

Equip $\mathbb{H}\otimes\mathbb{R}^{2^n}$ with the commuting complex structures $\mathcal{J}=\{\mathcal{J}_h\}_{h=1}^n$, where $\mathcal{J}_h(q\otimes e_K)=(-1)^{|K\cap\{h\}|}q\otimes e_{K\Delta\{h\}}$, for every $q\in\mathbb{H}$, $K\in\mathcal{P}(n)$, then define for every $h=1,...,n$ the operators
\begin{equation*}
    \partial_h:=\dfrac{1}{2}\left(\dfrac{\partial }{\partial \alpha_h}-\mathcal{J}_h\dfrac{\partial }{\partial\beta_h}\right),\qquad \overline{\partial}_h:=\dfrac{1}{2}\left(\dfrac{\partial }{\partial \alpha_h}+\mathcal{J}_h\dfrac{\partial }{\partial\beta_h}\right).
\end{equation*}
A stem function $F\in\mathcal{C}^1(D)$ is $h$-holomorphic when $\overline{\partial}_hF=0$ and holomorphic if it is $h$-holomorphic for every $h=1,...,n$. By \cite[Lemma 3.12]{Several}, $F=\sum_{K\in\mathcal{P}(n)}e_KF_K$ is $h$-holomorphic if and only if it satisfies the following Cauchy-Riemann system
\begin{equation*}
        \dfrac{\partial F_K}{\partial\alpha_h}=\dfrac{\partial F_{K\cup\{h\}}}{\partial\beta_h},\qquad
        \dfrac{\partial F_K}{\partial\beta_h}=-\dfrac{\partial F_{K\cup\{h\}}}{\partial\alpha_h},\qquad \forall K\in\mathcal{P}(n),h\notin K.
\end{equation*}
A slice regular function is a slice function induced by a holomorphic stem function. The set of slice regular functions over $\Omega_D$ will be denoted by $\mathcal{S}\mathcal{R}(\Omega_D)$. Since $\partial_h F$ and $\overline{\partial}_hF$ are stem functions if $F$ is so, we can define 
\begin{equation*}
    \dfrac{\partial f}{\partial x_h}:=\mathcal{I}(\partial_hF),\qquad \dfrac{\partial f}{\partial x_h^c}:=\mathcal{I}(\overline{\partial}_hF).
\end{equation*}
By \cite[Proposition 3.13]{Several}, $f\in\mathcal{S}\mathcal{R}(\Omega_D)$ if and only if $\partial f/\partial x_h^c=0$, for every $h=1,...,n$. Thus, $(\mathcal{S}\mathcal{R}(\Omega_D),\odot)$ is a real subalgebra of $(\mathcal{S}(\Omega_D),\odot)$, as consequence of Leibniz rule \cite[Proposition 3.25]{Several}.

Fix $h\in\{1,...,n\}$ and for any $y=(y_1,...,y_n)\in\Omega_D$, let $\Omega_{D,h}(y):=\{x\in\mathbb{H}\mid (y_1,...,y_{h-1},x,\linebreak y_{h+1},...,y_n)\in\Omega_D\}.$
We say that a function $f\in\mathcal{S}(\Omega_D)$ is slice (resp. slice regular) w.r.t. $x_h$ if for every $y\in\Omega_D$ its restriction $f^y_h:\Omega_{D,h}(y)\subset\mathbb{H}\rightarrow\mathbb{H}$, $f^y_h:x\mapsto f(y_1,...,y_{h-1},x,y_{h+1},...,y_n)$ is a one-variable slice (resp. slice regular) function. 
We call $f$ circular w.r.t. $x_h$ if $f^y_h$ is constant over the sets $\mathbb{S}_{\alpha,\beta}$ for any fixed $\alpha,\beta\in\mathbb{R}$ and $y\in\Omega_D$.
With $\mathcal{S}_h(\Omega_D)$, $\mathcal{S}\mathcal{R}_h(\Omega_D)$ and $\mathcal{S}_{c,h}(\Omega_D)$ we will denote respectively the subset of $\mathcal{S}(\Omega_D)$ of function that are slice, slice regular and circular w.r.t. $x_h$. For every $H\in\mathcal{P}(n)$, we also denote $\mathcal{S}_H(\Omega_D):=\bigcap_{h\in H}\mathcal{S}_h(\Omega_D)$, $\mathcal{S}\mathcal{R}_H(\Omega_D):=\bigcap_{h\in H}\mathcal{S}\mathcal{R}_h(\Omega_D)$ and $\mathcal{S}_{c,H}(\Omega_D):=\bigcap_{h\in H}\mathcal{S}_{c,h}(\Omega_D)$. In \cite[Proposition 3.1, 3.2, 3.4]{Parteteorica} those sets are characterized:
\begin{prop}
For any $H\in\mathcal{P}(n)$, it holds
\begin{equation*}
    \mathcal{S}_H(\Omega_D)=\left\{\mathcal{I}(F): F\in Stem(D), F=\sum_{K\subset H^c}e_KF_K+\sum_{h\in H}e_{\{h\}}\sum_{Q\subset\{ h+1,...,n\}\setminus H}e_QF_{\{h\}\cup Q}\right\},
\end{equation*}
\begin{equation}
    \label{Equazione caratterizzazione slice regular H}
\mathcal{S}\mathcal{R}_H(\Omega_D)=\mathcal{S}_H(\Omega_D)\bigcap_{h\in H}\ker(\partial/\partial x_h^c)\subset\mathcal{S}_H(\Omega_D),
\end{equation}
\begin{equation*}
    \mathcal{S}_{c,H}(\Omega_D)=\left\{\mathcal{I}(F): F\in Stem(D), F=\sum_{K\subset H^c}e_KF_K\right\}\subset\mathcal{S}_H(\Omega_D).
\end{equation*}
\end{prop}

Given $f\in\mathcal{S}(\Omega_D)$, induced by $F=\sum_{K\in\mathcal{P}(n)}e_KF_K$, we define its $x_h$-spherical value $f^\circ_{s,h}=\mathcal{I}(F^\circ_h)$ and its $x_h$-spherical derivative $f'_{s,h}=\mathcal{I}(F'_h)$, where $F^\circ_h\in Stem(D)$ and $F'_h\in Stem (D\setminus\mathbb{R}_h)$ are defined by
\begin{equation*}
    F^\circ_h(z)=\displaystyle\sum_{K\in\mathcal{P}(n),h\notin K}e_KF_K(z),\qquad F'_h(z)=\beta_h^{-1}\displaystyle\sum_{K\in\mathcal{P}(n),h\notin K}e_KF_{K\cup\{h\}}(z),
\end{equation*}
where $z=(z_1,...,z_n)=(\alpha_1+i\beta_1,...,\alpha_n+i\beta_n)$ and $\beta_h=\operatorname{im}(z_h)\in\mathbb{R}$. 
If $H=\{h_1,h_2,...,h_p\}\in\mathcal{P}(n)$ define also $f^\circ_{s,H}:=(\cdots(f^\circ_{s,h_1})^\circ_{s,h_2}\cdots)^\circ_{s,h_p}\in\mathcal{S}(\Omega_D)$ and $f'_{s,H}:=(\cdots(f'_{s,h_1})'_{s,h_2}\cdots)'_{s,h_p}\in\mathcal{S}(\Omega_{D_H})$, where $\Omega_{D_H}=\Omega_D\setminus\mathbb{R}_H$. 
Note that these definitions are well posed, since they do not depend on the order of $\{h_j\}_{h=1}^p$ \cite[Lemma 4.1]{Parteteorica}. Moreover, $f^\circ_{s,H}=\mathcal{I}(F^\circ_H)$, $f'_{s,H}=\mathcal{I}(F'_H)$, with
\begin{equation*}
    F^\circ_H(z)=\sum_{K\subset H^c}e_KF_K(z),\qquad F'_H(z)=\beta_H^{-1}\sum_{K\subset H^c}e_KF_{K\cup H}(z),
\end{equation*}
where $\beta_H:=\prod_{h\in H}\beta_h$. Note that, if $f\in\mathcal{S}^1(\Omega_D)$, then $f'_{s,H}$ is defined over all $\Omega_D$, for every $H\in\mathcal{P}(n)$. 
 In the next Proposition we recap the main properties of the spherical values and the spherical derivatives we will need through the paper. Their proofs can be found in \cite[\S 4]{Parteteorica}.
 
 \begin{prop}
\label{Proposizione richiami parte teorica}
Let $f,g\in\mathcal{S}(\Omega_D)$, $h\in\{1,...,n\}$ and $H\in\mathcal{P}(n)$, then the following holds true.
\begin{enumerate}
\item 
    \begin{equation}
    \label{equazione rappresentazione slice}
        f=f^\circ_{s,h}+\operatorname{Im}(x_h)\odot f'_{s,h},
    \end{equation} 
    where $\operatorname{Im}(x_h):\mathbb{H}^n\to\mathbb{H}$ is the slice function $(\alpha_1+J_1\beta_1,...,\alpha_n+J_n\beta_n)\mapsto J_h\beta_h$; 
    \item Leibniz formula
    \begin{equation}
    \label{equazione formula di Leibniz slice}
        (f\odot g)'_{s,h}=f'_{s,h}\odot g^\circ_{s,h}+f^\circ_{s,h}\odot g'_{s,h}.
    \end{equation}
    In particular, if $g\in\mathcal{S}_{c,H}(\Omega_D)$, then $(f\odot g)'_{s,H}=f'_{s,H}\odot g$.
    \item $f'_{s,H}\in\mathcal{S}_{c,H}(\Omega_{D_H})\cap\mathcal{S}_p(\Omega_{D_H})$, where $p=\min H^c=\min\{\{1,...,n\}\setminus H\}$;
    \item if $f\in\mathcal{S}_{c,h}(\Omega_D)$, then $f^\circ_{s,h}=f$ and $f'_{s,h}=0$;
    \item if $h\in H$, $H\cap\{1,...,h-1\}\neq\emptyset$ and $f\in\mathcal{S}_h(\Omega_D)$, then $f'_{s,H}=0$;
    \item if $f\in\ker(\partial/\partial x_t^c)$ for some $t\neq h$, then $f'_{s,h}\in\ker(\partial/\partial x_t^c)$;
    \item if $f\in\ker(\partial/\partial x_h^c)$, then $\Delta_hf'_{s,h}=0$, where $\Delta_h:=\frac{\partial^2}{\partial\alpha^2_h}+\frac{\partial^2}{\partial\beta^2_h}+\frac{\partial^2}{\partial\gamma^2_h}+\frac{\partial^2}{\partial\delta^2_h}$ is the Laplacian of $\mathbb{R}^4$ w.r.t. the variable $x_h=\alpha_h+i\beta_h+j\gamma_h+k\delta_h$.
\end{enumerate}
\end{prop}
We can give (\ref{equazione rappresentazione slice}) and (\ref{equazione formula di Leibniz slice}) through stem functions: let $F,G\in Stem(D)$ and $h\in\{1,...,n\}$, then it holds
\begin{equation}
     \label{equazione rappresentazione stem}
         F=F_h^\circ+\operatorname{Im}(Z_h)\otimes F'_h,
\end{equation}
where $\operatorname{Im}(Z_h)=\mathcal{I}^{-1}(\operatorname{Im}(x_h))\in Stem(\mathbb{C}^n)$, $\operatorname{Im}(Z_h)(z_1,\dots,z_n)=e_h\beta_h$, if $z_h=\alpha_h+i\beta_h$ and
\begin{equation}
     \label{equazione formula di Leibniz stem}
         (F\otimes G)'_h=F'_h\otimes G^\circ_h+F^\circ_h\otimes G'_h.
\end{equation}

Let us represent any $x=(x_1,...,x_n)\in\mathbb{H}^n$ with real coordinates as $x_h=\alpha_h+i\beta_h+j\gamma_h+k\delta_h$, then recall two differential operators for any $h=1,...,n$
\begin{equation*}
    \partial_{x_h}:=\dfrac{1}{2}\left(\dfrac{\partial}{\partial\alpha_h}-i\dfrac{\partial}{\partial\beta_h}-j\dfrac{\partial}{\partial\gamma_h}-k\dfrac{\partial}{\partial\delta_h}\right),\qquad
    \overline{\partial}_{x_h}:=\dfrac{1}{2}\left(\dfrac{\partial}{\partial\alpha_h}+i\dfrac{\partial}{\partial\beta_h}+j\dfrac{\partial}{\partial\gamma_h}+k\dfrac{\partial}{\partial\delta_h}\right).
\end{equation*}
They extend to several variables the well known Cauchy-Riemann-Fueter operators $\partial_{CRF}$ and $\overline{\partial}_{CRF}$ \cite[\S 6]{Harmonicity}. Sometimes notation can be slightly different and the factor $1/2$ can be omitted.  
These operators factorize the Laplacian, indeed $\forall h=1,...,n$
\begin{equation}
\label{equazione fattorizzazione laplaciano}
    4\partial_{x_h}\overline{\partial}_{x_h}=4\overline{\partial}_{x_h}\partial_{x_h}=\Delta_h.
\end{equation}
Functions in the kernel of $\overline{\partial}_{x_h}$ are usually called monogenic (or Fueter regular) w.r.t. $x_h$. By \eqref{equazione fattorizzazione laplaciano}, we get that $x_h$-monogenic functions are, in particular, harmonic w.r.t. $x_h$.
Furthermore we call $x_h$-axially monogenic those slice functions which are monogenic w.r.t. $x_h$, too, i.e.  $\mathcal{A}\mathcal{M}_h(\Omega_D):=\{f\in\mathcal{S}(\Omega_D)\mid f\in\mathcal{C}^1(\Omega_D), \ \overline{\partial}_{x_h}f=0\}$.
Finally, $\overline{\partial}_{x_h}f$ and $f'_{s,h}$ are closely related, whenever $f\in\mathcal{S}\mathcal{R}_h(\Omega_D)$, indeed by \cite[Lemma 4.3]{Parteteorica} it holds 
\begin{equation}
\label{Equazione relazione operatore crf e derivata sferica}
    \overline{\partial}_{x_h}f=-f'_{s,h}.
\end{equation}

\section{Almansi-type decomposition}

\subsection{Main results}

\begin{defn}
\label{Definizione S}
Let $\Omega_D\subset\mathbb{H}^n$, $f\in\mathcal{S}(\Omega_D)$ and $H\in\mathcal{P}(n)$. For every $K=\{k_1,...,k_p\}\subset \mathbb{H}$, with $k_1<...<k_p$, define over $\Omega_{D_H}$ the slice functions
\begin{equation*}
    \mathcal{S}^H_K(f):=\left(x_K\odot f\right)'_{s,H}=\left(\prod_{i=1}^px_{k_i}\odot f\right)'_{s,H},
\end{equation*}
and set $\mathcal{S}^\emptyset_\emptyset(f):=f$. If $H=\llbracket m\rrbracket:=\{1,2,...,m\}$ is an integers interval from 1 to some $m\in\{1,...,n\}$, we can write $\forall K\in\mathcal{P}(m)$
\begin{equation*}
    \mathcal{S}^{\llbracket m\rrbracket}_K(f):=(x_m^{\chi_K(m)}\dots(x_1^{\chi_K(1)}f)'_{s,1}\dots)'_{s,m},
    \end{equation*}
    where $\chi_K$ is the characteristic function of the set $K$. 
Note that, in this case, we can use the ordinary pointwise product as well as the slice product \cite[Proposition 2.52]{Several}.
If $f=\mathcal{I}(F)$, every $\mathcal{S}^H_K(f)$ is induced by the stem function
\begin{equation*}
    G^H_K(F):=\left(Z_K\otimes F\right)'_H,
\end{equation*}
where $Z_j\in Stem(\mathbb{C}^n)$ is the stem function $Z_j(\alpha_1+i\beta_1,...,\alpha_n+i\beta_n):=\alpha_j+e_j\beta_j$, inducing the monomial $x_j\in\mathcal{S}(\mathbb{H}^n)$, for any $j=1,...,n$.
\end{defn}

We can now formulate our main result.

\begin{thm}
\label{Teorema principale}
Let $\Omega_D\subset\mathbb{H}^n$ be a circular set and let $f\in\mathcal{S}(\Omega_D)$ be a slice function. Fix any $H\in\mathcal{P}(n)$, then
\begin{enumerate}
    \item we can decompose $f$ as
    \begin{equation}
    \label{Formula decomposizione di Almansi}
        f(x)=\displaystyle\sum_{K\subset H}(-1)^{|H\setminus K|}\left(\overline{x}\right)_{H\setminus K}\odot\mathcal{S}^H_K(f)(x).
    \end{equation}
    \end{enumerate}
\begin{enumerate}[resume]
    \item $\mathcal{S}^H_K(f)\in\mathcal{S}_{c,H}(\Omega_{D_H})\cap\mathcal{S}_p(\Omega_{D_H})$, where $p=\min H^c$, $\forall K\subset H$;
    \item if $f\in\mathcal{S}\mathcal{R}(\Omega_D)$, then $\Delta_h\mathcal{S}^H_K(f)=0$, $\forall h\in H$, $\forall K\subset H$;
    \end{enumerate}
\begin{enumerate}[resume]
\item $f\in\mathcal{S}\mathcal{R}(\Omega_D)$ if and only if $\mathcal{S}^H_K(f)\in\mathcal{S}\mathcal{R}_p(\Omega_D)$, $\forall H\in\mathcal{P}(n)$, $K\subset H$, $p=\min H^c$;
\item $f\in\mathcal{S}_\mathbb{R}(\Omega_D)$ if and only if $\mathcal{S}^{\llbracket n\rrbracket}_K(f)$ is real valued, $\forall K\in\mathcal{P}(n)$.
\end{enumerate}

\end{thm}

\begin{oss}
    For any $H\in\mathcal{P}(n)$ we can define the $\mathbb{H}$-right linear operator
\begin{equation*}
    \mathcal{S}^H:\mathcal{S}(\Omega_D)\ni f\mapsto\left\{\mathcal{S}^H_K(f)\right\}_{K\subset H}\in\left(\mathcal{S}_{c,H}(\Omega_{D_H})\cap\mathcal{S}_p(\Omega_{D_H})\right)^{2^{|H|}},
\end{equation*}
where $p:=\min H^c$ and its restriction
\begin{equation*}
    \mathcal{S}^H:\mathcal{S}\mathcal{R}(\Omega_D)\rightarrow\left(\mathcal{S}_{c,H}(\Omega_D)\cap\mathcal{S}\mathcal{R}_p(\Omega_D)\bigcap_{h\in H}\ker\Delta_h\right)^{2^{|H|}}.
\end{equation*}
\end{oss}

We highlight the unique character of the decomposition, indeed for every choice of $H\in\mathcal{P}(n)$, the functions $\mathcal{S}^H_K(f)$ are the only $H$-circular functions that realize decomposition (\ref{Formula decomposizione di Almansi}).
\begin{prop}
\label{Proposizione unicità decomposizione}
    Let $f\in\mathcal{S}(\Omega_D)$ and fix $H\in\mathcal{P}(n)$. Suppose that there exist functions $\{h_K\}_{K\subset H}$ such that $h_K\in\mathcal{S}_{c,H}(\Omega_D)$, $\forall K\subset H$ and
    \begin{equation*}
        f(x)=\displaystyle\sum_{K\subset H}(-1)^{|H\setminus K|}\left(\overline{x}\right)_{H\setminus K}\odot h_K(x).
    \end{equation*}
Then $h_K=\mathcal{S}^H_K(f)$.
\end{prop}

From Theorem \ref{Teorema principale}, we emphasize the case in which $H=\llbracket m\rrbracket=\{1,2,...,m\}$, that leads to what we call an ordered decomposition of $f$. 
\begin{cor}
\label{Corollario Almansi ordinato}
Let $f\in\mathcal{S}(\Omega_D)$ and fix any $m\in\{1,...,n\}$, then we can orderly decompose $f$ as
    \begin{equation}
    \label{Formula decomposizione di Almansi ordinata}
        f(x)=\displaystyle\sum_{K\in\mathcal{P}(m)}(-1)^{|K^c|}\left(\overline{x}\right)_{K^c}\mathcal{S}^{\llbracket m\rrbracket}_K(f)(x),
    \end{equation}
    where $K^c=\{1,...,m\}\setminus K$. Moreover,
    \begin{enumerate}
    \item if $m<n$, $\mathcal{S}^{\llbracket m\rrbracket}_K(f)\in\mathcal{S}_{c,\{1,...,m\}}(\Omega_{D_{\llbracket m\rrbracket}})\cap\mathcal{S}_{m+1}(\Omega_{D_{\llbracket m\rrbracket}})$, for any $K\in\mathcal{P}(m)$, while $\mathcal{S}^{\llbracket n\rrbracket}_K(f)\in\mathcal{S}_{c,\{1,...,n\}}(\Omega_{D_{\llbracket n\rrbracket}})$;
    \item if $f\in\mathcal{S}\mathcal{R}(\Omega_D)$, then $\Delta_h\mathcal{S}^{\llbracket m\rrbracket}_K(f)=0$, $\forall h\leq m$, $\forall K\in\mathcal{P}(m)$.
        \end{enumerate}
\end{cor}

We point out that formula (\ref{Formula decomposizione di Almansi ordinata}) holds with the ordinary pointwise product \cite[Proposition 2.52]{Several}. On the contrary, in (\ref{Formula decomposizione di Almansi}) the slice product is necessary.

We now give a one variable interpretation of slice regularity in terms of partial slice regularity of the functions $\mathcal{S}^{\llbracket m\rrbracket}_K(f)$.

\begin{prop}
\label{Proposizione caratterizzazione slice regolarita con componenti ordinate}
Let $f\in\mathcal{S}(\Omega_D)$, then
$f\in\mathcal{S}\mathcal{R}(\Omega_D)$ if and only if $\mathcal{S}^{\llbracket m\rrbracket}_K(f)\in\mathcal{S}\mathcal{R}_{m+1}(\Omega_D)$, $\forall m=0,...,n-1$, $K\in\mathcal{P}(m)$. In that case we can write $\mathcal{S}^{\llbracket m\rrbracket}_K(f)$ as
\begin{equation}
\label{equazione scrittura componenti ordinate con CRF}
    \mathcal{S}^{\llbracket m\rrbracket}_K(f)=(-1)^m\overline{\partial}_{x_m}(x_m^{\chi_K(m)}\dots\overline{\partial}_{x_1}(x_1^{\chi_K(1)}f)\dots).
\end{equation}
\end{prop}

\begin{oss}
    The previous characterization resembles the one given in \cite[Theorem 3.23]{Several}, in which iterations of spherical values and spherical derivatives (also referred as truncated spherical derivatives $\mathcal{D}_\epsilon(f)$) have been used. It is easy to see that truncated spherical derivatives can be expressed as real combinations of the components $\mathcal{S}^{\llbracket m\rrbracket}_K(f)$ and viceversa, making the two characterizations equivalent.
\end{oss}

\begin{ex}
\begin{enumerate}
    \item Let $f\in\mathcal{S}\mathcal{R}(\mathbb{H}^2)$, $f(x_1,x_2):=x_1 x_2$. We can give $2^2$ decompositions of $f$ for $H=\emptyset,\{1\},\{1,2\},\{2\}$: let $x=(\alpha_1+J_1\beta_1,\alpha_2+J_2\beta_2)$, then
\begin{equation*}
    \begin{split}
        f(x)&=\mathcal{S}^\emptyset_\emptyset(f)(x)=f(x)\\
        &=\mathcal{S}^{\llbracket1\rrbracket}_{\{1\}}(f)(x)-\overline{x}_1\mathcal{S}^{\llbracket1\rrbracket}_\emptyset(f)(x)=2\alpha_1x_2-\overline{x}_1 x_2\\
        &=\mathcal{S}^{\llbracket 2\rrbracket}_{\{1,2\}}(f)(x)-\overline{x}_1\mathcal{S}^{\llbracket 2\rrbracket}_{\{2\}}(f)(x)-\overline{x}_2\mathcal{S}^{\llbracket 2\rrbracket}_{\{1\}}(f)(x)+\overline{x}_1\overline{x}_2\mathcal{S}^{\llbracket 2\rrbracket}_\emptyset(f)(x)\\
        &\qquad=4\alpha_1\alpha_2-2\alpha_2\overline{x}_1-2\alpha_1\overline{x}_1+\overline{x}_1\overline{x}_2\\
        &=\mathcal{S}^{\{2\}}_{\{2\}}(f)(x)-\overline{x}_2\odot\mathcal{S}^{\{2\}}_\emptyset(f)(x)=2\alpha_2x_1-\overline{x}_2\odot x_1=2\alpha_2x_1-x_1\overline{x}_2.
    \end{split}
\end{equation*}
Note that in the first three decompositions the slice product is not needed. On the contrary, the last one, corresponding to $H=\{2\}$ needs the slice product. Moreover, $\mathcal{S}^\emptyset_\emptyset(f)=f\in\mathcal{S}\mathcal{R}_1(\Omega_D)$ and $\mathcal{S}^{\llbracket1\rrbracket}_\emptyset(f),\mathcal{S}^{\llbracket1\rrbracket}_{\{1\}}(f)\in\mathcal{S}\mathcal{R}_2(\Omega_D)$, as $f\in\mathcal{S}\mathcal{R}(\Omega_D)$.
\item Let $g\in\mathcal{S}\mathcal{R}(\mathbb{H}^3)$, $g(x_1,x_2,x_3)=e^{x_1}x_2x_3^3$. Now we have $2^3$ decompositions for $H\in\mathcal{P}(3)$. Let $x=(\alpha_1+J_1\beta_1,\alpha_2+J_2\beta_2,\alpha_3+J_3\beta_3)$, so aside from the trivial decomposition corresponding to $H=\emptyset$, we have the ordered decompositions for $H=\{1\},\{1,2\},\{1,2,3\}$
\begin{align*}
    f(x)&=\mathcal{S}^{\llbracket1\rrbracket}_{\{1\}}(f)(x)-\overline{x}_1\mathcal{S}^{\llbracket1\rrbracket}_\emptyset(f)(x)=e^{\alpha_1}(\cos\beta_1+\alpha_1/\beta_1\sin\beta_1)x_2x_3^3-\overline{x}_1e^{\alpha_1}\sin(\beta_1)/\beta_1x_2x_3^3\\
    &=\mathcal{S}^{\llbracket 2\rrbracket}_{\{1,2\}}(f)(x)-\overline{x}_1\mathcal{S}^{\llbracket 2\rrbracket}_{\{2\}}(f)(x)-\overline{x}_2\mathcal{S}^{\llbracket 2\rrbracket}_{\{1\}}(f)(x)+\overline{x}_1\overline{x}_2\mathcal{S}^{\llbracket 2\rrbracket}_\emptyset(f)(x)\\
    &\qquad=e^{\alpha_1}(\cos\beta_1+\alpha_1/\beta_1\sin\beta_1)2\alpha_2x_3^3-\overline{x}_1e^{\alpha_1}\sin(\beta_1)/\beta_12\alpha_2x_3^3+\\
    &\qquad-\overline{x}_2e^{\alpha_1}(\cos\beta_1+\alpha_1/\beta_1\sin\beta_1)x_3^3 +\overline{x}_1\overline{x}_2e^{\alpha_1}\sin(\beta_1)/\beta_1x_3^3\\
    &=\mathcal{S}^{\llbracket 3\rrbracket}_{\{1,2,3\}}(f)(x)-\overline{x}_1\mathcal{S}^{\llbracket 3\rrbracket}_{\{2,3\}}(f)(x)-\overline{x}_2\mathcal{S}^{\llbracket 3\rrbracket}_{\{1,3\}}(f)(x)-\overline{x}_3\mathcal{S}^{\llbracket 3\rrbracket}_{\{1,2\}}(f)(x)+\\
    &\qquad+\overline{x}_1\overline{x}_2\mathcal{S}^{\llbracket 3\rrbracket}_{\{3\}}(f)(x)+\overline{x}_1\overline{x}_3\mathcal{S}^{\llbracket 3\rrbracket}_{\{2\}}(f)(x)+\overline{x}_2\overline{x}_3\mathcal{S}^{\llbracket 3\rrbracket}_{\{1\}}(f)(x)-\overline{x}_1\overline{x}_2\overline{x}_3\mathcal{S}^{\llbracket 3\rrbracket}_\emptyset(f)(x)\\
    &\qquad =e^{\alpha_1}(\cos\beta_1+\alpha_1/\beta_1\sin\beta_1)2\alpha_24\alpha_3(\alpha_3^2-\beta_3^2)-\overline{x}_1e^{\alpha_1}\sin(\beta_1)/\beta_12\alpha_24\alpha_3(\alpha_3^2-\beta_3^2)+\\
    &\qquad-\overline{x}_2e^{\alpha_1}(\cos\beta_1+\alpha_1/\beta_1\sin\beta_1)4\alpha_3(\alpha_3^2-\beta_3^2)+\\
    &\qquad-\overline{x}_3e^{\alpha_1}(\cos\beta_1+\alpha_1/\beta_1\sin\beta_1)2\alpha_2(3\alpha_3^2-\beta_3^2)+\overline{x}_1\overline{x}_2e^{\alpha_1}\sin(\beta_1)/\beta_14\alpha_3(\alpha_3^2-\beta_3^2)+\\
    &\qquad+\overline{x}_1\overline{x}_3e^{\alpha_1}\sin(\beta_1)/\beta_12\alpha_2(3\alpha_3^2-\beta_3^2)+\overline{x}_2\overline{x}_3e^{\alpha_1}(\cos\beta_1+\alpha_1/\beta_1\sin\beta_1)(3\alpha_3^2-\beta_3^2)+\\
    &\qquad-\overline{x}_1\overline{x}_2\overline{x}_3e^{\alpha_1}\sin(\beta_1)/\beta_1(3\alpha_3^2-\beta_3^2)
    \end{align*}
    and the remaining decompositions for $H=\{2\},\{3\},\{1,3\},\{2,3\}$
    
    \begin{align*}
            f(x)&=\mathcal{S}^{\{2\}}_{\{2\}}(f)(x)-\overline{x}_2\odot \mathcal{S}^{\{2\}}_\emptyset(f)(x)=e^{x_1}2\alpha_2x_3^3-\overline{x}_2\odot e^{x_1}x_3^3\\
    &=\mathcal{S}^{\{3\}}_{\{3\}}(f)(x)-\overline{x}_3\odot \mathcal{S}^{\{3\}}_\emptyset(f)(x)=e^{x_1}x_24\alpha_3(\alpha_3^2-\beta_3^2)-\overline{x}_3\odot e^{x_1}x_2(3\alpha_3^2-\beta_3^2)\\
    &=\mathcal{S}^{\{1,3\}}_{\{1,3\}}(f)(x)-\overline{x}_1\odot \mathcal{S}^{\{1,3\}}_{\{3\}}(f)(x)-\overline{x}_3\odot \mathcal{S}^{\{1,3\}}_{\{1\}}(f)(x)+\overline{x}_1\overline{x}_3\odot \mathcal{S}^{\{1,3\}}_\emptyset(f)(x)\\
    &\qquad =e^{\alpha_1}(\cos\beta_1+\alpha_1/\beta_1\sin\beta_1)x_24\alpha_3(\alpha_3^2-\beta_3^2)-\overline{x}_1\odot e^{\alpha_1}\sin(\beta_1)/\beta_1x_24\alpha_3(\alpha_3^2-\beta_3^2)+\\
    &\qquad -\overline{x}_3\odot e^{\alpha_1}(\cos\beta_1+\alpha_1/\beta_1\sin\beta_1)x_2(3\alpha_3^2-\beta_3^2)+\overline{x}_1\overline{x}_3\odot e^{\alpha_1}\sin(\beta_1)/\beta_1x_2(3\alpha_3^2-\beta_3^2)\\
    &=\mathcal{S}^{\{2,3\}}_{\{2,3\}}(f)(x)-\overline{x}_2\odot \mathcal{S}^{\{2,3\}}_{\{3\}}(f)(x)-\overline{x}_3\odot \mathcal{S}^{\{2,3\}}_{\{2\}}(f)(x)+\overline{x}_2\overline{x}_3\odot \mathcal{S}^{\{2,3\}}_\emptyset(f)(x)\\
    &\qquad =e^{x_1}2\alpha_24\alpha_3(\alpha_3^2-\beta_3^2)-\overline{x}_2\odot e^{x_1}4\alpha_3(\alpha_3^2-\beta_3^2)-\overline{x}_3\odot e^{x_1}2\alpha_2(3\alpha_3^2-\beta_3^2)+\\
    &\qquad+\overline{x}_2\overline{x}_3\odot e^{x_1}(3\alpha_3^2-\beta_3^2).
\end{align*}
\end{enumerate}
\end{ex}

\subsection{Preliminary results}

\begin{lem}
\label{Lemma operazioni monomi}
For every $m=1,...,n$, it holds
\begin{enumerate}
    \item $(x_m)'_{s,h}=(Z_m)'_h=\delta_{h,m}$ and $(\overline{x}_m)'_{s,h}=(\overline{Z}_m)'_h=-\delta_{h,m}$, where $\delta_{i,j}$ is the Kronecker symbol and $\overline{Z}_m=\mathcal{I}^{-1}(\overline{x}_m)\in Stem(\mathbb{C}^n)$, $\overline{Z}_m(\alpha_1+i\beta_1,\dots,\alpha_n+i\beta_n)=\alpha_m-e_m\beta_m$;
    \item $(x_m)^\circ_{s,h}=x_m$, $(\overline{x}_m)^\circ_{s,h}=\overline{x}_m$, $(Z_m)^\circ_h=Z_m$, $(\overline{Z}_m)^\circ_h=\overline{Z}_m$ if $h\neq m$ and $(x_m)^\circ_{s,m}=(\overline{x}_m)^\circ_{s,m}=\operatorname{Re}(x_m)$, $(Z_m)^\circ_m=(\overline{Z}_m)^\circ_m=\operatorname{Re}(Z_m)$, where $\operatorname{Re}(Z_m)(\alpha_1+i\beta_1,\dots,\alpha_n+i\beta_n)=\alpha_m$ and $\operatorname{Re}(x_m)=\mathcal{I}(\operatorname{Re}(Z_m))$.
\end{enumerate}
\end{lem}

\begin{proof}
Since $x_m=\mathcal{I}(Z_m)$ and $\overline{x}_m=\mathcal{I}(\overline{Z}_m)$, it is enough to prove the properties for the slice funtions or for the stem functions. By Proposition \ref{Proposizione richiami parte teorica} (4), since $x_m,\overline{x}_m\in\mathcal{S}_{c,h}(\Omega_D)$, $\forall h\neq m$, immediately it holds $(x_m)^\circ_{s,h}=x_m$, $(\overline{x}_m)^\circ_{s,h}=\overline{x}_m$ and $(x_m)'_{s,h}=(\overline{x}_m)'_{s,h}=0$. Finally, by direct computation, $(Z_m)^\circ_m(z)=(\overline{Z}_m)^\circ_m(z)=\alpha_m=\operatorname{Re}(Z_m)(z)$, if $z=(\alpha_1+i\beta_1,\dots,\alpha_n+i\beta_n)$ and $(Z_m)'_m(z)=\operatorname{im}(z_m)^{-1}\operatorname{im}(z_m)=1$, $(\overline{Z}_m)'_m(z)=\operatorname{im}(z_m)^{-1}(-\operatorname{im}(z_m))=-1$.
\end{proof}

\begin{prop}
\label{Proposizione proprieta stem}
Let $F\in Stem(D)$ and fix $H\in\mathcal{P}(n)$. For every $K\subset H$, the functions $G^H_K(F)$ satisfy the following properties:
\begin{enumerate}
    \item For every $m\notin H$, it holds
    \begin{equation}
    \label{Proposizione proprieta stem eq 1}
        G^H_K(F)=G^{H\cup\{m\}}_{K\cup\{m\}}(F)-\overline{Z}_m\otimes G^{H\cup\{m\}}_K(F).
    \end{equation}
    \item For every $h\in H$ it holds
    \begin{equation*}
        G^H_K(F)=\left(Z_h^{\chi_K(h)}\otimes G^{H\setminus\{h\}}_{K\setminus\{h\}}(F)\right)'_h.
    \end{equation*}
    \item Explicitely, for $z=(z_1,...,z_n)\in D\setminus\mathbb{R}_H$, with $z_j=\alpha_j+i\beta_j$, we have 
    \begin{equation}
    \label{equazione esplicita delle componenti}
        G^H_K(F)(z)=\sum_{T\subset H^c}e_T\left(\displaystyle\sum_{L\subset K}\alpha_{K\setminus L}\beta^{-1}_{H\setminus L}F_{(T\cup H)\setminus L}(z)\right).
        \end{equation}
        \item If $F$ is holomorphic, then every $G^H_K(F)$ is $m$-holomorphic, $\forall m\notin H$.
        \end{enumerate}
\end{prop}

\begin{proof}
1. Apply (\ref{equazione rappresentazione stem}), (\ref{equazione formula di Leibniz stem}) and Lemma \ref{Lemma operazioni monomi}, then
\begin{equation*}
    \begin{split}
        &G^{H\cup\{m\}}_{K\cup\{m\}}(F)-\overline{Z}_m\otimes G^{H\cup\{m\}}_K(F)=\left(Z_{K\cup\{m\}}\otimes F\right)'_{H\cup\{m\}}-\overline{Z}_m\otimes\left(Z_K\otimes F\right)'_{H\cup\{m\}}=\\
        &=\left(\left(Z_{K\cup\{m\}}\otimes F\right)'_m\right)'_H-\overline{Z}_m\otimes\left(\left(Z_K\otimes F\right)'_m\right)'_H\\
        &=\left(\left(Z_{K\cup\{m\}}\right)'_m\otimes F^\circ_m+\left(Z_{K\cup\{m\}}\right)^\circ_m\otimes F'_m\right)'_H-\overline{Z}_m\otimes\left(\left(Z_K\right)'_m\otimes F^\circ_m+(Z_K)^\circ_m\otimes F'_m\right)'_H\\
        &=\left(Z_K\otimes F^\circ_m+(Z_m)^\circ_m\otimes Z_K\otimes F'_m\right)'_H-\overline{Z}_m\otimes\left(Z_K\otimes F'_m\right)'_H\\
        &=\left(Z_K\otimes F^\circ_m+\operatorname{Re}(Z_m)\otimes Z_K\otimes F'_m-\overline{Z}_m\otimes Z_K\otimes F'_m\right)'_H\\
        &=\left(Z_K\otimes\left(F^\circ_m+(\operatorname{Re}(Z_m)-\overline{Z}_m)\otimes F'_m\right)\right)'_H=\left(Z_K\otimes\left(F^\circ_m+\operatorname{Im}(Z_m)\otimes F'_m\right)\right)'_H\\
        &=\left(Z_K\otimes F\right)'_H=G^H_K(F).
    \end{split}
\end{equation*}
2. It follows immediately by definition of $G^H_K(F)$.\\
3. We procede by induction over $|H|$. Suppose first $|H|=1$, i.e. $H=\{h\}$ for some $h\in\{1,...,n\}$, then we have two components $G^{\{h\}}_\emptyset(F)$ and $G^{\{h\}}_{\{h\}}(F)$. Let us compute them explicitely: for any $z=(\alpha_1+i\beta_1,\dots,\alpha_n+i\beta_n)\in D\setminus\mathbb{R}_h$
\begin{equation*}
    G^{\{h\}}_\emptyset(F)(z)=F'_h(z):=\displaystyle\sum_{T\in\mathcal{P}(n),h\notin T}e_T\beta_h^{-1}F_{T\cup\{h\}}(z),
\end{equation*}
\begin{equation*}
    \begin{split}
        G^{\{h\}}_{\{h\}}(F)(z)&=\left(Z_h\otimes F\right)'_h(z)=\left(Z_h\right)'_h(z)\otimes F^\circ_h(z)+\left(Z_h\right)^\circ_h(z)\otimes F'_h(z)\\
        &=\displaystyle\sum_{T\in\mathcal{P}(n),h\notin T}e_TF_T(z)+\sum_{T\in\mathcal{P}(n),h\notin T}e_T\alpha_h\beta^{-1}_hF_{T\cup\{h\}}(z).
    \end{split}
\end{equation*}
Now, suppose that (\ref{equazione esplicita delle componenti}) holds for some $H\in\mathcal{P}(n)$ and let us prove it for $H'=H\cup\{m\}$, for any $m\notin H$. Suppose first $m\notin K$, then for $z=(z_1,...,z_n)\in D\setminus\mathbb{R}_H$, with $z_j=\alpha_j+i\beta_j$, $j=1,\dots,n$ we have
\begin{equation*}
    \begin{split}
        G^{H'}_K(F)(z)&=\left(\left(Z_K\otimes F\right)'_H\right)'_m(z)=\left(\sum_{T\subset H^c}e_T\left(\displaystyle\sum_{L\subset K}\alpha_{K\setminus L}\beta^{-1}_{H\setminus L}F_{(T\cup H)\setminus L}(z)\right)\right)'_m\\
        &=\beta_m^{-1}\displaystyle\sum_{T\subset H^c\setminus\{m\}}e_T\left(\displaystyle\sum_{L\subset K}\alpha_{K\setminus L}\beta^{-1}_{H\setminus L}F_{(T\cup H\cup\{m\})\setminus L}(z)\right)\\
        &=\displaystyle\sum_{T\subset(H\cup\{m\})^c}e_T\left(\displaystyle\sum_{L\subset K}\alpha_{K\setminus L}\beta^{-1}_{(H\cup\{m\})\setminus L}F_{(T\cup H\cup\{m\})\setminus L}(z)\right).
    \end{split}
\end{equation*}
Suppose now $m\in K$, then
\begin{equation*}
    \begin{split}
        &G^{H'}_K(F)(z)=\left(Z_m(z)\otimes\left(Z_{K\setminus\{m\}}\otimes F(z)\right)'_H\right)'_m\\
        &=\left(Z_m(z)\otimes\sum_{T\subset H^c}e_T\left(\displaystyle\sum_{L\subset (K\setminus\{m\})}\alpha_{(K\setminus\{m\})\setminus L}\beta^{-1}_{H\setminus L}F_{(T\cup H)\setminus L}(z)\right)\right)'_m\\
        &=\left(\sum_{T\subset H^c}e_T\left(\displaystyle\sum_{L\subset (K\setminus\{m\})}\alpha_{(K\setminus\{m\})\setminus L}\beta^{-1}_{H\setminus L}F_{(T\cup H)\setminus L}(z)\right)\right)^\circ_m+\\
        &+\alpha_m\left(\sum_{T\subset H^c}e_T\left(\displaystyle\sum_{L\subset (K\setminus\{m\})}\alpha_{(K\setminus\{m\})\setminus L}\beta^{-1}_{H\setminus L}F_{(T\cup H)\setminus L}(z)\right)\right)'_m\\
        &=\sum_{T\subset (H^c\setminus\{m\})}e_T\left(\displaystyle\sum_{L\subset (K\setminus\{m\})}\alpha_{(K\setminus\{m\})\setminus L}\beta^{-1}_{H\setminus L}F_{(T\cup H)\setminus L}(z)\right)+\\
        &+\beta_m^{-1}\alpha_m\sum_{T\subset (H^c)\setminus\{m\}}e_T\left(\displaystyle\sum_{L\subset (K\setminus\{m\})}\alpha_{(K\setminus\{m\})\setminus L}\beta^{-1}_{H\setminus L}F_{(T\cup H\cup\{m\})\setminus L}(z)\right)\\
        &=\sum_{T\subset (H\cup\{m\})^c}e_T\left(\displaystyle\sum_{L\subset (K\setminus\{m\})}\alpha_{K\setminus(\{m\}\cup L)}\beta^{-1}_{(H\cup\{m\}\setminus (L\cup\{m\})}F_{(T\cup H\cup\{m\})\setminus (L\cup\{m\})}(z)\right)+\\
        &+\sum_{T\subset(H\cup\{m\})^c}e_T\left(\displaystyle\sum_{L\subset (K\setminus\{m\})}\alpha_{K\setminus L}\beta^{-1}_{(H\cup\{m\})\setminus L}F_{(T\cup H\cup\{m\})\setminus L}(z)\right)\\
        &=\displaystyle\sum_{T\subset H'}e_T\left(\displaystyle\sum_{L\subset K}\alpha_{K\setminus L}\beta^{-1}_{H'\setminus L}F_{(T\cup H')\setminus L}(z)\right).
    \end{split}
\end{equation*}
4. $F$ and $Z_K$ are holomorphic, so is $Z_K\otimes F$. Finally, $G^H_K(F)=(Z_K\otimes F)'_H$ is holomorphic w.r.t. $z_m$ for every $m\notin H$, by (6) of Proposition \ref{Proposizione richiami parte teorica}.
\end{proof}




Next Lemma will be used to prove 5. of Theorem \ref{Teorema principale}.
\begin{lem}
For every $H\in\mathcal{P}(n)$, every $K\subset H$ and every $z=(\alpha_1+i\beta_1,\dots,\alpha_n+i\beta_n)\in D\setminus\mathbb{R}_H$, it holds
\begin{equation}
\label{formula generica per slice preserving}
    \beta_K^{-1}\displaystyle\sum_{T\subset H^c}e_TF_{K\cup T}(z)=\displaystyle\sum_{T\subset H\setminus K}(-1)^{|T|}\alpha_TG^H_{H\setminus(K\cup T)}(F)(z),
\end{equation}
where, if $H=\{1,...,n\}$ we mean
\begin{equation}
\label{formula finale per slice preserving}
    \beta_K^{-1}F_K(z)=\displaystyle\sum_{T\subset K^c}(-1)^{|T|}\alpha_TG^{\{1,...,n\}}_{(K\cup T)^c}(F)(z).
\end{equation}
\end{lem}

\begin{proof}
Let us proceed by induction over $|H|$. First, suppose $H=\{h\}$, for any $h=1,...,n$, then $K=\emptyset,\{h\}$. If $K=\emptyset$, we have
\begin{equation*}
    \begin{split}
        &\displaystyle\sum_{T\subset\{h\}}(-1)^{|T|}\alpha_TG^{\{h\}}_{\{h\}\setminus T}(F)(z)=G^{\{h\}}_{\{h\}}(F)(z)-\alpha_hG^{\{h\}}_\emptyset(F)(z)=
    \left(Z_h\otimes F\right)'_h(z)-\alpha_h F'_h(z)\\
    &=F^\circ_h(z)+\alpha_hF'_h(z)-\alpha_hF'_h(z)=\displaystyle\sum_{T\in\{h\}^c}e_TF_T(z),
    \end{split}
\end{equation*}
where (\ref{equazione formula di Leibniz stem}) and Lemma \ref{Lemma operazioni monomi} has been used. If $K=\{h\}$, immediately we get
\begin{equation*}
    G^{\{h\}}_\emptyset(F)(z)=F'_h(z)=\beta_h^{-1}\displaystyle\sum_{T\subset\{h\}^c}e_TF_{\{h\}\cup T}(z).
\end{equation*}
Now, suppose that (\ref{formula generica per slice preserving}) holds for some $H\in\mathcal{P}(n)$ and let us prove it for $H'=H\cup\{m\}$, with any $m\notin H$ and any $K\subset H'$. 
Let us split $m\in K$ and $m\notin K$. If $m\in K$, let us set $K':=K\setminus\{m\}$, then we have
\begin{equation*}
    \begin{split}
        &\displaystyle\sum_{T\subset H'\setminus K}(-1)^{|T|}\alpha_TG^{H'}_{H'\setminus(K\cup T)}(F)(z)=\displaystyle\sum_{T\subset H\setminus K'}(-1)^{|T|}\alpha_T\left(Z_{H\setminus(K'\cup T)}(z)\otimes F(z)\right)'_{H\cup\{m\}}\\
        &=\left(\displaystyle\sum_{T\subset H\setminus K'}(-1)^{|T|}\alpha_T\left(Z_{H\setminus(K'\cup T)}(z)\otimes F\right)'_H(z)\right)'_m=\left(\displaystyle\sum_{T\subset H\setminus K'}(-1)^{|T|}\alpha_TG^H_{H\setminus(K'\cup T)}(F)(z)\right)'_m\\
        &=\left(\beta_{K'}^{-1}\displaystyle\sum_{T\subset H^c}e_TF_{K'\cup T}(z)\right)'_m=\beta_K^{-1}\displaystyle\sum_{T\subset(H')^c}e_TF_{K\cup T}(z).
    \end{split}
\end{equation*}
Finally, if $m\notin K$
\begin{equation*}
    \begin{split}
        &\displaystyle\sum_{T\subset H'\setminus K}(-1)^{|T|}\alpha_TG^{H'}_{H'\setminus(K\cup T)}(F)(z)=\displaystyle\sum_{T\subset H\setminus K}(-1)^{|T|}\alpha_TG^{H'}_{H'\setminus(K\cup T)}(F)(z)+\\
        &\qquad-\displaystyle\sum_{T\subset H\setminus K}(-1)^{|T|}\alpha_m\alpha_TG^{H'}_{H\setminus(K\cup T)}(F)(z)\\
        &=\displaystyle\sum_{T\subset H\setminus K}(-1)^{|T|}\alpha_T\left(\left(Z_m(z)\otimes(Z_{H\setminus (K\cup T)}(z)\otimes F(z))\right)'_m\right)'_H-\displaystyle\sum_{T\subset H\setminus K}(-1)^{|T|}\alpha_m\alpha_T \left(Z_{H\setminus(K\cup T)}(z)\otimes F(z)\right)'_{H\cup m}\\
        &=\displaystyle\sum_{T\subset H\setminus K}(-1)^{|T|}\alpha_T\left(\left(Z_{H\setminus(K\cup T)}(z)\otimes F(z)\right)^\circ_m+\alpha_m\left(Z_{H\setminus(K\cup T)}(z)\otimes F(z)\right)'_m\right)'_H+\\
        &\qquad-\displaystyle\sum_{T\subset H\setminus K}(-1)^{|T|}\alpha_m\alpha_T \left(Z_{H\setminus(K\cup T)}(z)\otimes F(z)\right)'_{H\cup m}\\
        &=\left(\displaystyle\sum_{T\subset H\setminus K}(-1)^{|T|}\alpha_T G_{H\setminus(K\cup T)}^H(F)(z)\right)^\circ_m=\left(\beta_K^{-1}\displaystyle\sum_{T\subset H^c}e_TF_{K\cup T}(z)\right)^\circ_m=\beta_K^{-1}\displaystyle\sum_{T\subset (H')^c}e_TF_{K\cup T}(z).
    \end{split}
\end{equation*}
\end{proof}



\subsection{Proofs of main results}

\begin{proof}[Proof of Theorem \ref{Teorema principale}]
\begin{enumerate}
\item Let us prove that decomposition \eqref{Formula decomposizione di Almansi} holds for stem functions, too, namely that for any $H\in\mathcal{P}(n)$ we have
\begin{equation}
\label{Formula scomposizione per stem}
F=\displaystyle\sum_{K\subset H}(-1)^{|H\setminus K|}\overline{Z}_{H\setminus K}\otimes G^H_K(F).
\end{equation}
We proceed by induction over $|H|$. Suppose first $H=\{h\}$, for some $h=1,...,n$, then, for any $z=(\alpha_1+i\beta_1,\dots,\alpha_n+i\beta_n)\in D$ we have
\begin{equation*}
\begin{split}
    &G^{\{h\}}_{\{h\}}(F)(z)-\overline{Z}_h(z)\otimes G^{\{h\}}_\emptyset(F)(z)=(Z_h(z)\otimes F(z))'_h-\overline{Z}_h(z)\otimes F'_h(z)\\
    &=F^\circ_h(z)+\alpha_h F'_h(z)-\alpha_h F'_h(z)+\operatorname{Im}(Z_h)(z)\otimes F'_h(z)=F(z),
    \end{split}
\end{equation*}
by (\ref{equazione formula di Leibniz slice}), (\ref{equazione rappresentazione stem}) and Lemma \ref{Lemma operazioni monomi}.
Now, suppose (\ref{Formula scomposizione per stem}) holds for some $H\in\mathcal{P}(n)$, let us prove it for $H'=H\cup\{m\}$, with $m\notin H$. We have
\begin{equation*}
    \begin{split}
        &\displaystyle\sum_{K\subset H'}(-1)^{|H'\setminus K|}\overline{Z}_{H'\setminus K}(z)\otimes G^{H'}_K(F)(z)\\
        &=
    \displaystyle\sum_{K\subset H}(-1)^{|H\setminus K|}\overline{Z}_{H\setminus K}(z)\otimes G^{H\cup\{m\}}_{K\cup\{m\}}(F)(z)-\displaystyle\sum_{K\subset H}(-1)^{|H\setminus K|}\overline{Z}_{H\setminus K}(z)\otimes\overline{Z}_m(z)\otimes G^{H\cup\{m\}}_K(F)(z)\\
    &=\displaystyle\sum_{K\subset H}(-1)^{|H\setminus K|}\overline{Z}_{H\setminus K}(z)\otimes\left(G^{H\cup\{m\}}_{K\cup\{m\}}(F)(z)-\overline{Z}_m(z)\otimes G^{H\cup\{m\}}_K(F)(z)\right)\\
    &=\displaystyle\sum_{K\subset H}(-1)^{|H\setminus K|}\overline{Z}_{H\setminus K}(z)\otimes G^H_K(F)(z)=F(z),
    \end{split}
\end{equation*}
by \eqref{Proposizione proprieta stem eq 1} and the inductive hypothesis. Now \eqref{Formula decomposizione di Almansi} easily follows, indeed
\begin{equation*}
    \begin{split}
        f&=\mathcal{I}(F)=\mathcal{I}\left(\displaystyle\sum_{K\subset H}(-1)^{|H\setminus K|}\overline{Z}_{H\setminus K}\otimes G^H_K(F)\right)\\
        &=\displaystyle\sum_{K\subset H}(-1)^{|H\setminus K|}\mathcal{I}\left(\overline{Z}_{H\setminus K}\right)\odot\mathcal{I}\left(G^H_K(F)\right)=\displaystyle\sum_{K\subset H}(-1)^{|H\setminus K|}\left(\overline{x}\right)_{H\setminus K}\odot\mathcal{S}^H_K(f).
    \end{split}
\end{equation*}
\item For any $K\subset H$, $\mathcal{S}^H_K(f)=(x_K\odot f)'_{s,H}\in\mathcal{S}_{c,H}(\Omega_{D_H})\cap\mathcal{S}_p(\Omega_{D_H})$, by Proposition \ref{Proposizione richiami parte teorica} (3).
\item Write $\mathcal{S}^H_K(f)=\left(x_h^{\chi_K(h)}\odot\mathcal{S}^{H\setminus\{h\}}_{K\setminus\{h\}}(f)\right)'_{s,h}$. By hypothesis, $f\in\ker(\partial/\partial x_t^c)$, $\forall t=1,...,n$, then, by 6. of Proposition \ref{Proposizione richiami parte teorica}, $\mathcal{S}^{H\setminus\{h\}}_{K\setminus\{h\}}(f)\in\ker(\partial/\partial x_h^c)$ and thanks to Leibniz formula \cite[Proposition 3.25]{Several}, $x_h^{\chi_K(h)}\odot\mathcal{S}^{H\setminus\{h\}}_{K\setminus\{h\}}(f)\in\ker(\partial/\partial x_h^c)$. Finally, by Proposition \ref{Proposizione richiami parte teorica} 7. $\Delta_h\mathcal{S}^H_K(f)=\Delta_h\left(x_h^{\chi_K(h)}\odot\mathcal{S}^{H\setminus\{h\}}_{K\setminus\{h\}}(f)\right)'_{s,h}=0.$
\item $\Rightarrow$) By hypothesis, $x_K\odot f\in\mathcal{S}\mathcal{R}(\Omega_D)$, then $\mathcal{S}^H_K(f)=(x_K\odot f)'_{s,H}\in\ker(\partial/\partial x_t^c)$, for any $t\notin H$. In particular, $\mathcal{S}^H_K(f)=(x_K\odot f)'_{s,H}\in\ker(\partial/\partial x_p^c)\cap\mathcal{S}_p(\Omega_D)=\mathcal{S}\mathcal{R}_p(\Omega_D)$, by (\ref{Equazione caratterizzazione slice regular H}).
    
     $\Leftarrow$) It is a particular case of Proposition \ref{Proposizione caratterizzazione slice regolarita con componenti ordinate}.
    \item $f$ is slice preserving if and only if $F_K$ is real $\forall K\in\mathcal{P}(n)$, which by (\ref{formula finale per slice preserving}) is equivalent for $G^{\{1,...,n\}}_K(F)=\mathcal{S}^{\llbracket n\rrbracket}_K(f)$ to be real valued for every $K\in\mathcal{P}(n)$.
    

\end{enumerate}
\end{proof}

\begin{proof}[Proof of Proposition \ref{Proposizione unicità decomposizione}]
Apply Proposition \ref{Proposizione richiami parte teorica} (2) and the hypothesis $h_T\in\mathcal{S}_{c,H}(\Omega_D)$ in the following computation
    \begin{equation*}
        \begin{split}
            \mathcal{S}^H_K(f):&=(x_K\odot f)'_{s,H}=\left(x_K\odot \displaystyle\sum_{T\subset H}(-1)^{|H\setminus T|}\left(\overline{x}\right)_{H\setminus T}\odot h_T\right)'_{s,H}\\
            &=\displaystyle\sum_{T\subset H}(-1)^{|H\setminus T|}\left(x_K\odot\left(\overline{x}\right)_{H\setminus T}\odot h_T\right)'_{s,H}\\
            &=\displaystyle\sum_{T\subset H}(-1)^{|H\setminus T|}\left(x_K\odot\left(\overline{x}\right)_{H\setminus T}\right)'_{s,H}\odot h_T.
        \end{split}
    \end{equation*}
    Now we claim that $\left(x_K\odot\left(\overline{x}\right)_{H\setminus T}\right)'_{s,H}=(-1)^{|H\setminus K|}\delta_{K,T}$, which would reduce the prevoious equation to $\mathcal{S}^H_K(f)=h_K$. Suppose first that exists $ h\in T\setminus K\subset H$, then $x_K\odot\left(\overline{x}\right)_{H\setminus T}\in\mathcal{S}_{c,h}(\Omega_D)$, thus in particular $\left(x_K\odot\left(\overline{x}\right)_{H\setminus T}\right)'_{s,H}=0$. Viceversa, suppose $h\in K\setminus T\subset H$, but again $x_K\odot\left(\overline{x}\right)_{H\setminus T}\in\mathcal{S}_{c,h}(\Omega_D)$, indeed $x_K\odot\left(\overline{x}\right)_{H\setminus T}=x_h\overline{x}_h\odot x_{K\setminus\{h\}}\odot\left(\overline{x}\right)_{H\setminus (T\cup\{h\})}=(\alpha_h^2+\beta_h^2)x_{K\setminus\{h\}}\odot\left(\overline{x}\right)_{H\setminus (T\cup\{h\})}\in\mathcal{S}_{c,h}(\Omega_D)$. So, the unique non trivial element of the sum refers to $T=K$, for which we have
    \begin{equation*}
        \begin{split}
            \left(x_K\odot\left(\overline{x}\right)_{H\setminus K}\right)'_{s,H}&=\left(\left(x_K\odot\left(\overline{x}\right)_{H\setminus K}\right)'_{s,K}\right)'_{s,H\setminus K}=\left((x_K)'_{s,K}\odot\left(\overline{x}\right)_{H\setminus K}\right)'_{s,H\setminus K}\\
            &=\left(\left(\overline{x}\right)_{H\setminus K}\right)'_{s,H\setminus K}=(-1)^{|H\setminus K|},
        \end{split}
    \end{equation*}
    where we have used Proposition \ref{Proposizione richiami parte teorica} (2) and Lemma \ref{Lemma operazioni monomi} (1).
\end{proof}

\begin{proof}[Proof of Proposition \ref{Proposizione caratterizzazione slice regolarita con componenti ordinate}]
\item[$\implies$)] We have already proved in Theorem \ref{Teorema principale} (4).

\item[$\impliedby$)] For any $m\in\{1,...,n\}$ consider the ordered Almansi-type decomposition (\ref{Formula decomposizione di Almansi ordinata}) of $f$. 
Recall that $\partial \overline{x}_{h}/\partial x_{k}^c=0$, $\forall h,k=1,...,n$ and $\partial/\partial x_{m+1}^c(\mathcal{S}^{\llbracket m\rrbracket}_K(f))=0$, for every $K\in\mathcal{P}(n)$, so applying \cite[(73)]{Several} it holds
        \begin{equation*}
            \begin{split}
                \dfrac{\partial f}{\partial x_{m+1}^c}&=\dfrac{\partial}{\partial x_{m+1}^c}\left(\displaystyle\sum_{K\in\mathcal{P}(m)}(-1)^{|K^c|}\left(\overline{x}\right)_{K^c}\mathcal{S}^{\llbracket m\rrbracket}_K(f)\right)=\displaystyle\sum_{K\in\mathcal{P}(m)}(-1)^{|K^c|}\left(\overline{x}\right)_{K^c}\otimes\dfrac{\partial\mathcal{S}^{\llbracket m\rrbracket}_K(f)}{\partial x_{m+1}^c}=0.
            \end{split}
        \end{equation*}
       This proves the charaterization of $f\in\mathcal{S}\mathcal{R}(\Omega_D)$.
        Finally, (\ref{equazione scrittura componenti ordinate con CRF}) follows from (\ref{Equazione relazione operatore crf e derivata sferica}). 
\end{proof}

\section{Applications}
The first application we give of Theorem \ref{Teorema principale} concerns quaternionic (ordered) polynomials with right coefficients, in which the components of the decomposition are given through zonal harmonics. Note that Theorem \ref{Teorema principale} can be fully applied, since every polynomial with right coefficients is a slice regular function \cite[Proposition 3.14]{Several}.
Before, we recall a result from \cite[Corollary 6.7]{Harmonicity}.
\begin{lem}
\label{lemma derivata potenza zonali}
For every $m\geq 0$, consider the slice regular power $x^m:\mathbb{H}\rightarrow\mathbb{H}$. Then it holds
\begin{equation*}
        (x^m)'_s=\tilde{\mathcal{Z}}_{m-1}(x),
\end{equation*}
where 
\begin{equation*}
    \tilde{\mathcal{Z}}_k(x):=\left\{\begin{array}{ll}
         \dfrac{\mathcal{Z}_k(x,1)}{k+1}&\text{ if }k\geq0  \\
         0& \text{ if }k=-1
    \end{array}\right.
\end{equation*}
and $\mathcal{Z}_k(x,1)$ is the real valued zonal harmonic of $\mathbb{R}^4$ with pole 1 (see \cite[Ch. 5]{HFT}).
\end{lem}

\begin{prop}
\label{Proposizione componenti caso polinomi}
Let $P\in\mathbb{H}[X_1,...,X_n]$ be any quaternionic polynomial with right coefficients
\begin{equation*}
    P(x_1,...,x_n)=\displaystyle\sum_{k=1}^n\displaystyle\sum_{|\alpha|=k}x^\alpha a_\alpha=\displaystyle\sum_{k=1}^n\displaystyle\sum_{|\alpha|=k}x_1^{\alpha_1}... \ x_n^{\alpha_n} a_\alpha.
\end{equation*}
Then, for every $H\in\mathcal{P}(n)$ and $K\subset H$
\begin{equation}
\label{equazione componenti caso polinomi}
    \mathcal{S}^H_K(P)(x)=\displaystyle\sum_{k=1}^n\displaystyle\sum_{|\alpha|=k}\prod_{j\in H}\tilde{\mathcal{Z}}_{\alpha_j-1+\chi_K(j)}(x_j)\prod_{i\in H^c} x^{\alpha_i}_{i}a_\alpha,
\end{equation}
where $\prod_ix_i^{\alpha_i}$ is an ordered product.
\end{prop}

\begin{proof}
By linearity of the spherical derivative, we can assume without loss of generality \linebreak$P(x_1,...,x_n)=x^\alpha=x_1^{\alpha_1}...\ x_n^{\alpha_n}$. We will proceed by induction over $|H|$. Suppose $H=\{h\}$, for some $h=1,...,n$, then, since $x_1^{\alpha_1}... \ x_{h-1}^{\alpha_{h-1}} x_{h+1}^{\alpha_{h+1}}... \ x_n^{\alpha_n}\in\mathcal{S}_{c,h}(\Omega_D)$, we have
\begin{equation*}
    \begin{split}
        \mathcal{S}^{\{h\}}_\emptyset(P)&=(x^\alpha)'_{s,h}=(x_h^{\alpha_h}\odot x_1^{\alpha_1}... \ x_{h-1}^{\alpha_{h-1}} x_{h+1}^{\alpha_{h+1}}... \ x_n^{\alpha_n})'_{s,h}\\&=(x_h^{\alpha_h})'_{s,h}\odot x_1^{\alpha_1}...\ x_{h-1}^{\alpha_{h-1}} x_{h+1}^{\alpha_{h+1}}...\  x_n^{\alpha_n}=\tilde{\mathcal{Z}}_{\alpha_h-1}(x_h) x_1^{\alpha_1}...\ x_{h-1}^{\alpha_{h-1}} x_{h+1}^{\alpha_{h+1}}...\ x_n^{\alpha_n},
    \end{split}
\end{equation*}
where we have used Lemma \ref{lemma derivata potenza zonali}, Proposition \ref{Proposizione richiami parte teorica} (2) and that $\title{Z}_j$ is real valued. Similarly,
    \begin{equation*}
        \mathcal{S}^{\{h\}}_{\{h\}}(P)=(x_h\odot x^\alpha)'_{s,h}=(x_h^{\alpha_h+1}\odot x_1^{\alpha_1}...\ x_{h-1}^{\alpha_{h-1}} x_{h+1}^{\alpha_{h+1}}... \ x_n^{\alpha_n})'_{s,h}=\tilde{\mathcal{Z}}_{\alpha_h}(x_h) x_1^{\alpha_1}...\ x_{h-1}^{\alpha_{h-1}} x_{h+1}^{\alpha_{h+1}}...\ x_n^{\alpha_n}.
    \end{equation*}
     Now, suppose that (\ref{equazione componenti caso polinomi}) holds for some $H\in\mathcal{P}(n)$ and let us prove it for $H'=H\cup\{m\}$, for any $m\notin H$. Suppose first $m\notin K$, then, as before
    \begin{equation*}
        \begin{split}
            \mathcal{S}^{H'}_K(P)&=(\mathcal{S}^H_K(P))'_{s,m}=\left(\prod_{j\in H}\tilde{\mathcal{Z}}_{\alpha_j-1+\chi_K(j)}(x_j)\prod_{i\in H^c}x_i^{\alpha_i}\right)'_{s,m}\\
            &=\left(\prod_{j\in H}\tilde{\mathcal{Z}}_{\alpha_j-1+\chi_K(j)}(x_j)x_m^{\alpha_m}\odot\prod_{i\in (H')^c}x_i^{\alpha_i}\right)'_{s,m}\\
            &=\prod_{j\in H}\tilde{\mathcal{Z}}_{\alpha_j-1+\chi_K(j)}(x_j)\tilde{\mathcal{Z}}_{\alpha_m-1}(x_m)\prod_{i\in (H')^c}x_i^{\alpha_i}=\prod_{j\in H'}\tilde{\mathcal{Z}}_{\alpha_j-1+\chi_K(j)}(x_j)\prod_{i\in (H')^c}x_i^{\alpha_i}.
        \end{split}
    \end{equation*}
    If $m\in K$, let $K'=K\setminus\{m\}$, then
\begin{equation*}
    \begin{split}
        \mathcal{S}^{H'}_K(P)&=(x_m\odot\mathcal{S}^H_{K'}(P))'_{s,m}=\left(\prod_{j\in H}\tilde{\mathcal{Z}}_{\alpha_j-1+\chi_{K'}(j)}(x_j)x_m^{\alpha_m+1}\odot\prod_{i\in (H')^c}x_i^{\alpha_i}\right)'_{s,m}\\
        &=\prod_{j\in H}\tilde{\mathcal{Z}}_{\alpha_j-1+\chi_{K'}(j)}(x_j)\tilde{\mathcal{Z}}_{\alpha_m}(x_m)\prod_{i\in (H')^c}x_i^{\alpha_i}=\prod_{j\in H'}\tilde{\mathcal{Z}}_{\alpha_j-1+\chi_{K}(j)}(x_j)\prod_{i\in (H')^c}x_i^{\alpha_i}.
    \end{split}
\end{equation*}
\end{proof}

Let us examine the case in which $f$ is slice w.r.t. $x_h$.
\begin{prop}
    Let $f\in\mathcal{S}_h(\Omega_D)$ for some $h\in\{1,...,n\}$, then
    \begin{equation}
    \label{quasi tutte le componenti sono nulle se h slice}
    \mathcal{S}^{\llbracket h\rrbracket}_K(f)=0,\qquad\forall K\in\mathcal{P}(h-1), K\neq\{1,...,h-1\}. 
    \end{equation}
    In particular, the ordered decomposition of $f$ of order $h$ reduces to
    \begin{equation*}
        f=\sum_{K\in\mathcal{P}(h-1)}(-1)^{|K^c|}\overline{x}_{K^c}\mathcal{S}^{\llbracket h\rrbracket}_{K\cup\{h\}}(f)-\overline{x}_h\mathcal{S}^{\llbracket h\rrbracket}_{\{1,...,h-1\}}(f).
    \end{equation*}
\end{prop}

\begin{proof}
Assume $K\in\mathcal{P}(h-1)$, with $K\neq\{1,...,h-1\}$, then there exists $m\in\{1,...,h-1\}\setminus K$. Let $H:=\{m,h\}$, then by (5) of Proposition \ref{Proposizione richiami parte teorica} it holds $f'_{s,H}=0$. Since $K\cap H=\emptyset$, $x_K\in\mathcal{S}_{c,H}(\Omega_D)$, then by (2) of Proposition \ref{Proposizione richiami parte teorica}  we have
    \begin{equation*}
        \mathcal{S}^{\llbracket h\rrbracket}_K(f)=\left[(x_K\odot f)'_{s,H}\right]'_{s,\{1,...,h\}\setminus H}=(x_K\odot f'_{s,H})'_{s,\{1,...,h\}\setminus H}=0.
    \end{equation*}
This proves \eqref{quasi tutte le componenti sono nulle se h slice}. By this and Corollary \ref{Corollario Almansi ordinato} follows
    \begin{equation*}
        \begin{split}
            f&=\sum_{K\in\mathcal{P}(h)}(-1)^{|\{1,...,h\}\setminus K|}\overline{x}_{\{1,...,h\}\setminus K}\mathcal{S}^{\llbracket h\rrbracket}_{K}(f)=\sum_{K\in\mathcal{P}(h-1)}(-1)^{|\{1,...,h\}\setminus K|}\overline{x}_{\{1,...,h\}\setminus K}\mathcal{S}^{\llbracket h\rrbracket}_{K}(f)+\\
            &+\sum_{K\in\mathcal{P}(h-1)}(-1)^{|\{1,...,h\}\setminus (K\cup\{h\})|}\overline{x}_{\{1,...,h\}\setminus (K\cup\{h\})}\mathcal{S}^{\llbracket h\rrbracket}_{K\cup\{h\}}(f)\\
            &=-\overline{x}_h\mathcal{S}^{\llbracket h\rrbracket}_{\{1,...,h-1\}}(f)+\sum_{K\in\mathcal{P}(h-1)}(-1)^{|\{1,...,h-1\}\setminus K|}\overline{x}_{\{1,...,h-1\}\setminus K}\mathcal{S}^{\llbracket h\rrbracket}_{K\cup\{h\}}(f).
        \end{split}
    \end{equation*}
\end{proof}

The components of the ordered decomposition provide examples of axially monogenic functions.
\begin{prop}
\label{proposizione componenti ordinate monogeniche}
    Let $f\in\mathcal{S}\mathcal{R}(\Omega_D)$, $m=1,...,n-1$ and recall that $\mathcal{A}\mathcal{M}_m(\Omega_D)$ denotes the set of slice functions monogenic with respect to $x_m$. Then $\forall K\in\mathcal{P}(m)$, $\mathcal{S}^{\llbracket m\rrbracket}_K(f)$ satisfies
    \begin{enumerate}
        \item $\partial_{x_m}(\mathcal{S}^{\llbracket m\rrbracket}_K(f))\in\mathcal{A}\mathcal{M}_m(\Omega_D)$;
        \item $\Delta_{m+1}(\mathcal{S}^{\llbracket m\rrbracket}_K(f))\in\mathcal{A}\mathcal{M}_{m+1}(\Omega_D)$.
    \end{enumerate}
\end{prop}

\begin{proof}
    \begin{enumerate}
    By Corollary \ref{Corollario Almansi ordinato}, $\mathcal{S}^{\llbracket h\rrbracket}_K(f)$ is harmonic w.r.t. $x_j$, for any $j=1,...,h$, $\forall K\in\mathcal{P}(h)$, so
        \begin{equation*}
            \overline{\partial}_{x_m}\partial_{x_m}\left(\mathcal{S}^{\llbracket m\rrbracket}_K(f)\right)=\dfrac{1}{4}\Delta_m\left(\mathcal{S}^{\llbracket m\rrbracket}_K(f)\right)=0
        \end{equation*}
        and by Proposition \ref{Proposizione caratterizzazione slice regolarita con componenti ordinate}
        \begin{equation*}
            \overline{\partial}_{x_{m+1}}\Delta_{m+1}\left(\mathcal{S}^{\llbracket m\rrbracket}_K(f)\right)=\Delta_{m+1}\overline{\partial}_{x_{m+1}}\left(\mathcal{S}^{\llbracket m\rrbracket}_K(f)\right)=-\Delta_{m+1}\left(\mathcal{S}^{\llbracket m+1\rrbracket}_K(f)\right)=0.
        \end{equation*}
    \end{enumerate}
\end{proof}

The following result highlights a difference between the one and several variables slice regular functions: in the first case the Laplacian of a slice regular function is always an axially monogenic functions (this is Fueter's Theorem \cite{Fueter}), in the latter this happens only for the first variable. But
for any variable, we can at least write it as sum of axially monogenic functions.
\begin{lem}
\label{lemma laplaciano}
Let $m=1,...,n$ and let $f\in\mathcal{S}^1(\Omega_D)\cap\ker(\partial/\partial x_m^c)$, then it holds
\begin{equation*}
    \Delta_mf=-4\displaystyle\sum_{K\in\mathcal{P}(m-1)}(-1)^{|K^c|}\left(\overline{x}\right)_{K^c}\partial_{x_m}\left(\mathcal{S}^{\llbracket m\rrbracket}_K(f)\right).
\end{equation*}
\end{lem}

\begin{proof}
By (3) and (6) of Proposition \ref{Proposizione richiami parte teorica} it holds $\mathcal{S}^{\llbracket m-1\rrbracket}_K(f)=(x_K\odot f)'_{s,\{1,...,m-1\}}\in\ker(\partial/\partial x_m^c)\cap\mathcal{S}_m(\Omega_D)=\mathcal{S}\mathcal{R}_m(\Omega_D)$, $\forall K\in\mathcal{P}(m-1)$, then by (\ref{equazione fattorizzazione laplaciano}) and (\ref{Equazione relazione operatore crf e derivata sferica}) it holds
\begin{equation*}
    \Delta_m\mathcal{S}^{\llbracket m-1\rrbracket}_K(f)=4\partial_{x_m}\overline{\partial}_{x_m}\mathcal{S}^{\llbracket m-1\rrbracket}_K(f)=-4\partial_{x_m}\left[\left(\mathcal{S}^{\llbracket m-1\rrbracket}_K(f)\right)'_{s,m}\right]=-4\partial_{x_m}\left(\mathcal{S}^{\llbracket m\rrbracket}_K(f)\right),
\end{equation*}
with $\partial_{x_m}\left(\mathcal{S}^{\llbracket m\rrbracket}_K(f)\right)\in\mathcal{A}\mathcal{M}_m(\Omega_D)$, by Proposition \ref{proposizione componenti ordinate monogeniche}, 1.
So, applying (\ref{Formula decomposizione di Almansi ordinata}), we have
\begin{equation*}
    \begin{split}
        \Delta_m f&=\Delta_m\left(\displaystyle\sum_{K\in\mathcal{P}(m-1)}(-1)^{|K^c|}(\overline{x})_{K^c}\mathcal{S}^{\llbracket m-1\rrbracket}_K(f)\right)=\displaystyle\sum_{K\in\mathcal{P}(m-1)}(-1)^{|K^c|}(\overline{x})_{K^c}\Delta_m\mathcal{S}^{\llbracket m-1\rrbracket}_K(f)\\
        &=-4\displaystyle\sum_{K\in\mathcal{P}(m-1)}(-1)^{|K^c|}(\overline{x})_{K^c}\partial_{x_m}\left(\mathcal{S}^{\llbracket m\rrbracket}_K(f)\right).
    \end{split}
\end{equation*}

\end{proof}

The issue changes if we assume the function slice regular in that specific variable, as already proven in \cite[Theorem 4.9]{Parteteorica}, getting a generalization of Fueter's Theorem in several variables. We give another proof through the ordered decomposition of Corollary \ref{Corollario Almansi ordinato}.
\begin{cor}
\label{Corollario fueter theorem in several variables}
    Let $f\in\mathcal{S}\mathcal{R}_m(\Omega_D)$, then $\Delta_mf\in\mathcal{A}\mathcal{M}_m(\Omega_D)$.
\end{cor}

\begin{proof}
    By Proposition \ref{proposizione componenti ordinate monogeniche}, Lemma \ref{lemma laplaciano} and \eqref{quasi tutte le componenti sono nulle se h slice}  we get
    \begin{equation*}
        \Delta_mf=-4\displaystyle\sum_{K\in\mathcal{P}(m-1)}(-1)^{|K^c|}(\overline{x})_{K^c}\partial_{x_m}\left(\mathcal{S}^{\llbracket m\rrbracket}_K(f)\right)=-4\partial_{x_m}\left(\mathcal{S}^{\llbracket m\rrbracket}_{\{1,...,m-1\}}(f)\right)\in\mathcal{A}\mathcal{M}_m(\Omega_D).
    \end{equation*}
\end{proof}

\begin{cor}
\label{Corollario biarmonicita}
Every slice regular function is separately biharmonic in each variable.
\end{cor}

\begin{proof}
Let $f\in\mathcal{S}\mathcal{R}(\Omega_D)$, then thanks to Corollary \ref{Corollario Almansi ordinato} 2, Lemma \ref{lemma laplaciano} and (\ref{equazione fattorizzazione laplaciano}) we have
\begin{equation*}
\begin{split}
    \Delta_m^2f&=\Delta_m\left(-4\displaystyle\sum_{K\in\mathcal{P}(m-1)}(-1)^{|\{1,...,m-1\}\setminus K|}\overline{x}_{\{1,...,m-1\}\setminus K}\partial_{x_m}\left(\mathcal{S}^{\llbracket m\rrbracket}_K(f)\right)\right)\\
    &=-4\displaystyle\sum_{K\in\mathcal{P}(m-1)}(-1)^{|\{1,...,m-1\}\setminus K|}\overline{x}_{\{1,...,m-1\}\setminus K}\partial_{x_m}\left(\Delta_m\left(\mathcal{S}^{\llbracket m\rrbracket}_K(f)\right)\right)=0.
\end{split}
\end{equation*}
\end{proof}

\subsection{Mean value properties for slice regular functions}
\label{Sezione integrali}

\emph{Notation}: in the last part of the section, if $a=(a_1,...,a_n)\in\mathbb{H}^n$ and $\lambda_k\in\mathbb{H}$, for some $k=1,...,n$, we denote $a+\lambda_k:=(a_1,...,a_{k-1},a_k+\lambda_k,a_{k+1},...,a_n)$.
Let $\sigma$ be the surface measure of $\mathbb{S}^3=\partial\mathbb{B}^4\subset\mathbb{H}\cong\mathbb{R}^4$ such that $\sigma\left(\mathbb{S}^3\right)=1$, namely $\sigma(y):=\mathcal{H}^3(y)/\omega_3$, where $\mathcal{H}^3$ denotes the three-dimensional Hausdorff measure of $\mathbb{R}^4$ and $\omega_3:=\mathcal{H}^3\left(\mathbb{S}^3\right)=2\pi^2$. Again, by $\sigma^l$ we mean the $l$-th power of $\sigma$. 
 
 \emph{Assumption}: throughout the section we will always assume that $f\in\mathcal{S}\mathcal{R}(\Omega_D)$ is a slice regular function and for $a=(a_1,...,a_n)\in\Omega_D$ and $r_1,...,r_n\in\mathbb{R}^+\cup\{0\}$ it holds $\overline{B_{r_1}(a_1)}\times...\times \overline{B_{r_n}(a_n)}\subset\Omega_D$.

\begin{prop}
\label{prop mean value formula for components}
Let $f\in\mathcal{S}\mathcal{R}(\Omega_D)$, then for every $H\in\mathcal{P}(n)$ and $K\subset H$ it holds
\begin{equation}
\label{mean value formula for components}
    (\mathcal{S}^H_K(f))(a)=\int_{\left(\mathbb{S}^3\right)^{|H|}}\left(\mathcal{S}^H_K(f)\right)\left(a+\textstyle\sum_{h\in H}r_h\lambda_h\right)d\sigma^{|H|}(\lambda).
\end{equation}
\end{prop}
\begin{proof}
By Theorem \ref{Teorema principale}, every $\mathcal{S}^H_K(f)$ is separately harmonic w.r.t. $x_h$, for every $h\in H$, thus, we can apply the classical mean value formula for harmonic functions for such variables
\begin{equation*}
    \left(\mathcal{S}^H_K(f)\right)(a)=\int_{\mathbb{S}^3}\left(\mathcal{S}^H_K(f)\right)(a+r_h\lambda_h)d\sigma(\lambda_h),\qquad\forall h\in H.
\end{equation*}
Thus, (\ref{mean value formula for components}) follows applying the previous formula for any $h\in H$.
\end{proof}

\begin{prop}[First mean value formula]
\label{prop mean value formula k}
For any $m=1,...,n$ it holds
\begin{equation}
\label{Equazione formula integrale domini costanti}
\begin{split}
f(a)=\sum_{K\in\mathcal{P}(m)}(-1)^{|K^c|}\overline{a}_{K^c}\int_{\left(\mathbb{S}^3\right)^m}\mathcal{S}^{\llbracket m\rrbracket}_K(f)(a+\textstyle\sum_{i=1}^mr_m\lambda_m)d\sigma^m(\lambda).
\end{split}
\end{equation}
\end{prop}

\begin{proof}
Apply (\ref{Formula decomposizione di Almansi ordinata}) and (\ref{mean value formula for components}) with $H=\{1,...,m\}$
\begin{equation*}
    f(a)=\displaystyle\sum_{K\in\mathcal{P}(m)}(-1)^{|K^c|}\overline{a}_{K^c}\mathcal{S}^{\llbracket m\rrbracket}_K(f)(a)=\sum_{K\in\mathcal{P}(m)}(-1)^{|K^c|}\overline{a}_{K^c}\int_{\left(\mathbb{S}^3\right)^m}\mathcal{S}^{\llbracket m\rrbracket}_K(f)(a+\textstyle\sum_{i=1}^mr_i\lambda_i)d\sigma^m(\lambda).
\end{equation*}
\end{proof}

We can give integral formulas through the general decomposition (\ref{Formula decomposizione di Almansi}), but we must assume that the centers of the spheres are real.
\begin{prop}
Let $H\in\mathcal{P}(n)$ and assume $a_h\in\mathbb{R}$, for any $h\in H$. Then
\begin{equation*}
    f(a)=\sum_{K\subset H}(-1)^{|H\setminus K|}a_{H\setminus K}\int_{(\mathbb{S})^{|H|}}\mathcal{S}^H_K(f)(a+\textstyle\sum_{h\in H}r_h\lambda_h)d\sigma^{|H|}(\lambda).
\end{equation*}
\end{prop}

\begin{proof}
By (\ref{Formula decomposizione di Almansi}) and (\ref{mean value formula for components}) we have
\begin{equation*}
    \begin{split}
        f(a)&=\sum_{K\subset H}(-1)^{|H\setminus K|}\left[\overline{x}_{H\setminus K}\odot\mathcal{S}^H_K(f)\right](a)=\sum_{K\subset H}(-1)^{|H\setminus K|}a_{H\setminus K}\mathcal{S}^H_K(f)(a)=\\
    &=\sum_{K\subset H}(-1)^{|H\setminus K|}a_{H\setminus K}\int_{\left(\mathbb{S}^3\right)^{|H|}}\left(\mathcal{S}^H_K(f)\right)\left(a+\textstyle\sum_{h\in H}r_h\lambda_h\right)d\sigma^{|H|}(\lambda),
    \end{split}
\end{equation*}
where we have used that $a_h\in\mathbb{R}$, $\forall h\in H$.
\end{proof}

We give another integral formula through decomposition (\ref{Formula decomposizione di Almansi ordinata}). For $m\geq2$, it highly differs from (\ref{Equazione formula integrale domini costanti}), beacause of the components involved and the dimension of the domain of integration. On the contrary, they coincide if $m=1$.
\begin{prop}[Second mean value formula]
\label{prop mean value formula vuoto}
For any $m=1,...,n$ it holds
\begin{equation}
\label{Equazione prima formula della media}
\begin{split}
    f(a)&=\displaystyle\sum_{j=0}^{m-1}r_{\{1,...,j\}}\int_{\left(\mathbb{S}^3\right)^{j+1}}\overline{\lambda}_{\{1,...,j\}}\mathcal{S}^{\llbracket j\rrbracket}_\emptyset(f)(a+\textstyle\sum_{i=1}^{j+1}r_i\lambda_i)d\sigma^{j+1}(\lambda)+\\
    &+r_{\{1,...,m\}}\int_{\left(\mathbb{S}^3\right)^m}\overline{\lambda}_{\{1,...,m\}}\mathcal{S}^{\llbracket m\rrbracket}_\emptyset(f)(a+\textstyle\sum_{i=1}^mr_i\lambda_i)d\sigma^m(\lambda),
\end{split}
\end{equation}
where $r_\emptyset=\overline{\lambda}_\emptyset=1$.

\end{prop}
\begin{proof}
We prove the identity by induction over $m$, using the corresponding one-variable formula \cite[Proposition 2]{Almansi} that we can apply iteratively, since $\mathcal{S}^{\llbracket m\rrbracket}_K(f)\in\mathcal{S}\mathcal{R}_{m+1}(\Omega_D)$. For $m=1$, (\ref{Equazione prima formula della media}) is precisely \cite[Proposition 2]{Almansi}, indeed
\begin{equation*}
\begin{split}
    f(a)&=\int_{\mathbb{S}^3}f(a+r_1\lambda_1)d\sigma(\lambda_1)+r_1\int_{\mathbb{S}^3}\overline{\lambda}_1\left(\mathcal{S}^{\llbracket1\rrbracket}_\emptyset(f)\right)(a+r_1\lambda_1)d\sigma(\lambda_1)\\
    &=\int_{\mathbb{S}^3}f(a+r_1\lambda_1)d\sigma(\lambda_1)+r_1\int_{\mathbb{S}^3}\overline{\lambda}_1f'_{s,1}(a+r_1\lambda_1)d\sigma(\lambda_1).
\end{split}
\end{equation*}
{\allowdisplaybreaks
Now, suppose that the formula holds for some $m$, then, apply \cite[Proposition 2]{Almansi} to $\mathcal{S}^{\llbracket m\rrbracket}_\emptyset(f)^{\tilde{a}}_{m+1}\in\mathcal{S}\mathcal{R}(\Omega_{D,m+1}(\tilde{a}))$, where $\tilde{a}=a+\textstyle\sum_{i=1}^mr_i\lambda_i$:

\begin{equation*}
    \begin{split}
    f(a)&=\displaystyle\sum_{j=0}^{m-1}r_{\{1,...,j\}}\int_{\left(\mathbb{S}^3\right)^{j+1}}\overline{\lambda}_{\{1,...,j\}}\mathcal{S}^{\llbracket j\rrbracket}_\emptyset(f)(a+\textstyle\sum_{i=1}^{j+1}r_i\lambda_i)d\sigma^{j+1}(\lambda)+\\
    &+r_{\{1,...,m\}}\int_{\left(\mathbb{S}^3\right)^m}\overline{\lambda}_{\{1,...,m\}}\mathcal{S}^{\llbracket m\rrbracket}_\emptyset(f)(a+\textstyle\sum_{i=1}^{m}r_i\lambda_i)d\sigma^m(\lambda)\\
    &=\displaystyle\sum_{j=0}^{m-1}r_{\{1,...,j\}}\int_{\left(\mathbb{S}^3\right)^{j+1}}\overline{\lambda}_{\{1,...,j\}}\mathcal{S}^{\llbracket j\rrbracket}_\emptyset(f)(a+\textstyle\sum_{i=1}^{j+1}r_i\lambda_i)d\sigma^{j+1}(\lambda)+\\
        &+r_{\{1,...,m\}}\int_{\left(\mathbb{S}^3\right)^m}\overline{\lambda}_{\{1,...,m\}}\left[\int_{\mathbb{S}^3}\mathcal{S}^{\llbracket m\rrbracket}_\emptyset(f)(a+\textstyle\sum_{i=1}^{m+1}r_i\lambda_i)d\sigma(\lambda_{m+1})\right]d\sigma^m(\lambda)+\\
        &+r_{\{1,...,m\}}\int_{\left(\mathbb{S}^3\right)^m}\overline{\lambda}_{\{1,...,m\}}\left[r_{m+1}\int_{\mathbb{S}^3}\overline{\lambda}_{m+1}\mathcal{S}^{\llbracket m+1\rrbracket}_\emptyset(f)(a+\textstyle\sum_{i=1}^{m+1}r_i\lambda_i)d\sigma(\lambda_{m+1})\right]d\sigma^m(\lambda)
        \end{split}
        \end{equation*}
        \begin{equation*}
            \begin{split}
    &=\displaystyle\sum_{j=0}^{m-1}r_{\{1,...,j\}}\int_{\left(\mathbb{S}^3\right)^{j+1}}\overline{\lambda}_{\{1,...,j\}}\mathcal{S}^{\llbracket j\rrbracket}_\emptyset(f)(a+\textstyle\sum_{i=1}^{j+1}r_i\lambda_i)d\sigma^{j+1}(\lambda)+\\
        +&r_{\{1,...,m\}}\int_{\left(\mathbb{S}^3\right)^{m+1}}\overline{\lambda}_{\{1,...,m\}}\mathcal{S}^{\llbracket m\rrbracket}_\emptyset(f)(a+\textstyle\sum_{i=1}^{m+1}r_i\lambda_i)d\sigma^{m+1}(\lambda)+\\
        &+r_{\{1,...,m+1\}}\int_{\left(\mathbb{S}^3\right)^{m+1}}\overline{\lambda}_{\{1,...,m+1\}}\mathcal{S}^{\llbracket m+1\rrbracket}_\emptyset(f)(a+\textstyle\sum_{i=1}^{m+1}r_i\lambda_i)d\sigma^{m+1}(\lambda)\\
        &=\sum_{j=0}^{m}r_{\{1,...,j\}}\int_{\left(\mathbb{S}^3\right)^{j+1}}\overline{\lambda}_{\{1,...,j\}}\mathcal{S}^{\llbracket j\rrbracket}_\emptyset(f)(a+\textstyle\sum_{i=1}^{j+1}r_i\lambda_i)d\sigma^{j+1}(\lambda)+\\
        &+r_{\{1,...,m+1\}}\int_{\left(\mathbb{S}^3\right)^{m+1}}\overline{\lambda}_{\{1,...,m+1\}}\mathcal{S}^{\llbracket m+1\rrbracket}_\emptyset(f)(a+\textstyle\sum_{i=1}^{m+1}r_i\lambda_i)d\sigma^{m+1}(\lambda).
    \end{split}
\end{equation*}}
\end{proof}


In the rest of the section we mimic what has been done so far, but with the Poisson kernel: first we find Poisson formulas for the components $\mathcal{S}^{\llbracket m\rrbracket}_K(f)$ and finally two types of formulas for $f$.

\begin{prop}
\label{Poisson formula components}
Let $x_1,...,x_m\in\mathbb{B}_\mathbb{H}$, then it holds
\begin{equation}
\label{Equazione formula di Poisson per componenti}
    \begin{split}
        \mathcal{S}^{\llbracket m\rrbracket}_K(f)(a+\textstyle\sum_{i=1}^mr_ix_i)=\displaystyle\int_{(\mathbb{S}^3)^m}\mathcal{S}^{\llbracket m\rrbracket}_K(f)(a+\textstyle\sum_{i=1}^mr_i\xi_i)\prod_{j=1}^mP(x_j,\xi_j)d\sigma^m(\xi),
    \end{split}
\end{equation}
where $P(x,\xi):=\dfrac{1-|x|^2}{|x-\xi|^4}$ is the Poisson kernel of $\mathbb{B}\subset\mathbb{R}^4$.
\end{prop}
\begin{proof}
By Corollary \ref{Corollario Almansi ordinato}, every $\mathcal{S}^{\llbracket m\rrbracket}_K(f)$ is harmonic w.r.t. $x_1,...,x_m$, so by Poisson integral formula for harmonic functions it holds for any $k=1,...,m$
\begin{equation*}
        \mathcal{S}^{\llbracket m\rrbracket}_K(f)(a+\textstyle\sum_{i=1}^mr_ix_i)=\displaystyle\int_{\mathbb{S}^3}\mathcal{S}^{\llbracket m\rrbracket}_K(f)(a+\textstyle\sum_{i\neq k}r_ix_i+r_k\xi_k)P(x_k,\xi_k)d\sigma(\xi_k).
\end{equation*}
Thus, (\ref{Equazione formula di Poisson per componenti}) follows by applying the previous formula for $k=1,...,m$.
\end{proof}
\begin{prop}[First Poisson formula]
Let $m=1,...,n$ and $x_1,...,x_m\in\mathbb{B}_\mathbb{H}$, then it holds
\begin{equation*}
f(a+\textstyle\sum_{i=1}^mr_ix_i)=\displaystyle\sum_{K\in\mathcal{P}(m)}(-1)^{|K^c|}\left(\overline{a}+r\overline{x}\right)_{K^c}\int_{(\mathbb{S}^3)^m}\mathcal{S}^{\llbracket m\rrbracket}_K(f)(a+\textstyle\sum_{i=1}^mr_i\xi_i)\prod_{j=1}^mP(x_j,\xi_j)d\sigma^m(\xi).
\end{equation*}
\end{prop}
\begin{proof}
Apply (\ref{Formula decomposizione di Almansi ordinata}) and  (\ref{Equazione formula di Poisson per componenti}) to every $\mathcal{S}^{\llbracket m\rrbracket}_K(f)$, to get
\begin{equation*}
\begin{split}
    f(a+\textstyle\sum_{i=1}^mr_ix_i)&=\displaystyle\sum_{K\in\mathcal{P}(m)}(-1)^{|K^c|}\left(\overline{a}+r\overline{x}\right)_{K^c}\mathcal{S}^{\llbracket m\rrbracket}_K(a+\textstyle\sum_{i=1}^mr_ix_i)\\
    &=\displaystyle\sum_{K\in\mathcal{P}(m)}(-1)^{|K^c|}\left(\overline{a}+r\overline{x}\right)_{K^c}\int_{(\mathbb{S}^3)^m}\mathcal{S}^{\llbracket m\rrbracket}_K(f)(a+\textstyle\sum_{i=1}^mr_i\xi_i)\prod_{j=1}^mP(x_j,\xi_j)d\sigma^m(\xi).
    \end{split}
\end{equation*}
\end{proof}

\begin{prop}[Second Poisson formula]
Let $m=1,...,n$ and $x_1,...,x_m\in\mathbb{B}_\mathbb{H}$, then it holds
\begin{equation*}
    \begin{split}
        f(a+\textstyle\sum_{i=1}^mr_ix_i)&=\displaystyle\sum_{j=1}^{m-1}r_{\{1,...,j\}}\int_{(\mathbb{S}^3)^{j+1}}(\overline{\xi}-\overline{x})_{\{1,...,j\}}(\mathcal{S}^{\llbracket j\rrbracket}_{\emptyset}(f))(a+\textstyle\sum_{i=1}^{j+1}r_i\xi_i)\prod_{t=1}^{j+1}P(x_t,\xi_t)d\sigma^{j+1}(\xi)+\\
        &+r_{\{1,...,m\}}\int_{(\mathbb{S}^3)^m}(\overline{\xi}-\overline{x})_{\{1,...,m\}}\mathcal{S}^{\llbracket m\rrbracket}_\emptyset(f)(a+\textstyle\sum_{i=1}^{m}r_i\xi_i)\prod_{t=1}^{m}P(x_t,\xi_t)d\sigma^{m}(\xi).
    \end{split}
\end{equation*}
\end{prop}
\begin{proof}
The proof is analogue of the one of Proposition \ref{prop mean value formula vuoto}, but here apply \cite[Proposition 3]{Almansi}.
\end{proof}

\printbibliography
\end{document}